\documentclass[12pt]{article}
\usepackage{amsfonts}
\usepackage{amssymb, amscd, amsthm}
\usepackage{latexsym}
\usepackage{multirow}
\usepackage{amsmath}
\usepackage[all]{xy}
\usepackage[dvips]{graphicx}
\usepackage{mathrsfs}
\usepackage{fancyhdr}

\newtheorem{lemma}[subsection]{Lemma}
\newtheorem{thm}[subsection]{Theorem}
\newtheorem{prop}[subsection]{Proposition}
\newtheorem{rem}[subsection]{Remark}
\newtheorem{coro}[subsection]{Corollary}
\newtheorem{defn}[subsection]{Definition}

\newcommand{\ra}{\rightarrow}
\newcommand{\mo}{\mathcal{O}}
\newcommand{\mf}{\mathcal{F}}

\newcommand{\me}{\mathcal{E}}
\newcommand{\mext}{\mathbb{E}\mathsf{x}\mathsf{t}}
\newcommand{\mhom}{\mathbb{H}\mathsf{o}\mathsf{m}}
\newcommand{\ms}{\mathcal{S}}
\newcommand{\mt}{\mathcal{T}}
\newcommand{\mh}{\mathcal{H}}
\newcommand{\mm}{\mathcal{M}}
\newcommand{\mn}{\mathcal{N}}
\newcommand{\mw}{\mathcal{W}}
\newcommand{\mv}{\mathcal{V}}
\newcommand{\um}{\mathcal{U}}
\newcommand{\ts}{\mathtt{S}}
\newcommand{\mc}{\mathcal{C}}
\newcommand{\mr}{\mathcal{R}}

\newcommand{\ls}{|dH|}
\newcommand{\p}{\mathbb{P}}

\newcommand{\bl}{\mathbb{L}}

\begin{document}
\fontsize{12pt}{14pt} \textwidth=14cm \textheight=21 cm
\numberwithin{equation}{section}
\title{Motivic measures of the moduli spaces of pure sheaves on $\mathbb{P}^2$ with all degrees.}
\author{Yao YUAN}
\date{}
\maketitle
\begin{flushleft}{\textbf{Abstract.}} Let $\mm(d,\chi)$ be the moduli stack of stable sheaves of rank 0, Euler characteristic $\chi$ and first Chern class $dH~(d>0)$, with $H$ the hyperplane class in $\mathbb{P}^2$.  We compute the $A$-valued motivic measure $\mu_A(\mm(d,\chi))$ of $\mm(d,\chi)$ and get explicit formula in codimension $D:=\rho_d-1$, where $\rho_d$ is $d-1$ for $d=p$ or $2p$ with $p$ prime, and $7$ otherwise.  As a corollary, we get the last $2(D+1)$ Betti numbers of the moduli scheme $M(d,\chi)$ when $d$ is coprime to $\chi$.   

\end{flushleft}

%%%%%%%%%%%%%%%%%%%%%%%%%%%%%%%%%%%%%%%%%%%%%%%%%%%%%%%%%%%%%%
%%%%%%%%%%%%%%%%         Section 1         %%%%%%%%%%%%%%%%%%%
\section{Introduction.}
The moduli space $M$ of 1-dimensional semistable sheaves on a surface is very interesting.  Sheaves in $M$ are supported at curves inside the surface.  Hence $M$ seems to be close to a Jacobian family.  Actually, properties of $M$ do sometimes give us some results on (compactified) Jacobians of curves of plannar singularities, such as Corollary 4.2.13 in \cite{yuan} and Corollary 7.6 in \cite{yyf}.  However, $M$ in general is far more complicated than a Jacobian family because there are sheaves supported at curves with very bad singularities (e.g. reducible, non-reduced). 

Many other people have worked on the moduli space $M$, such as \cite{jmdmm},\cite{lee} and \cite{ky}.  In particular, on a K3 or abelian surface,  the deformation equivalence classes of $M$ are known in a large generality by Yoshioka's work in \cite{ky}.     

Let $M(d,\chi)$ be the moduli scheme parametrizing 1-dimensional semistable sheaves on $\p^2$ with rank 0, first Chern class $dH$ for $H$ the hyperplane class, and Euler characteristic $\chi$.  The Pandharipande-Thomas theory defined in \cite{pt} on local 3-folds together with Toda's work in \cite{toda} give a prediction that the Euler number $e(M(d,\chi))$ does not depend on $\chi$ given $d,\chi$ coprime.  Also Physicists have computed $e(M(d,\chi))$ for $d\leq 300$ using their argument not mathematically correct (see Equation (4.2) and Table 4 in Section 8.3 in \cite{kkv}).  Despite that, there is no general explicit statement on $e(M(d,\chi))$,  Betti numbers $b_i(M(d,\chi))$, or Hodge numbers $h^{p,q}(M(d,\chi))$. 

Let $\mm(d,\chi)$ be the stack associated to the same moduli functor as $M(d,\chi)$.  Let $Hilb^{n}(\p^2)$ ($\mh^n$ resp.) be the moduli space at scheme (stack resp.) level of ideal sheaves of colength $n$ on $\p^2$.  Let $\mu_A(-)$ be some $A$-valued motivic measure with $A$ a commutative ring or a field if needed.  Let $A_m$ be the subgroup generated by $\mu_A(\ms)$ with $dim~\ms\leq m.$  Let $\bl:=\mu_A(\mathbb{A})$ with $\mathbb{A}$ the affine line. 

In this paper, we prove the following theorem.
\begin{thm}[Theorem \ref{main}]For and $d>0$ and $\chi$, let $\chi_0\equiv \pm\chi~mod~(d)$ and $-\frac {3d}2\leq\chi_0\leq-d$ (such $\chi_0$ is unique).  Then  we have
\[\mu_A(\mm(d,\chi))~\equiv~\bl^{3d+1+2\chi_0}\cdot \mu_A(\mh^{\bar{d}}),~~~mod~(A_{d^2-\rho_d}),\]
with $\bar{d}=\frac{d(d-3)}2-\chi_0$ and 
\[\rho_d=\left\{\begin{array}{l}d-1,~ for ~d=p~or~2p~with ~p~ prime.\\ 7,~~~~~otherwise.\end{array}\right.\]  

On the scheme level we have
\[\mu_A(M(d,\chi))~\equiv~\bl^{3d+1+2\chi_0}\cdot \mu_A(Hilb^{\bar{d}}(\p^2)),~~~mod~(A_{d^2-\rho_d+1}).\]
\end{thm}
We then have three corollaries as follows.
\begin{coro}[Corollary \ref{bet}]Let $b_i(-)$ and $h^{p,q}(-)$ be the $i$-th Betti number and Hodge number with index $(p,q)$ respectively.  Then for any $d>0$ and $\chi$ coprime to $d$, if $i$ and $p+q$ are both no less than $1+2(d^2+1-\rho_d)$, we  then have

(1) $b_{i}(M(d,\chi))=0$ for $i$ odd.

(2) $h^{p,p}(M(d,\chi))=b_{2p}(M(d,\chi))=b_{2p-2(3d+1+2\chi_0)}(Hilb^{\bar{d}}(\p^2))$. 

(3) $h^{p,q}=0$ for $p\neq q$.
\end{coro}
\begin{coro}[Corollary \ref{disk}]For any $d>0$ and $\chi_1,\chi_1$, we have
\[\mu_A(\mm(d,\chi_1))~\equiv~ \mu_A(\mm(d,\chi_2)),~~~mod~(A_{d^2-\rho_d}).\]
In particular, if $\chi_i$ are coprime to $d$ for $i=1,2$, then we have
 \[\mu_A(M(d,\chi_1))~\equiv~ \mu_A(M(d,\chi_2)),~~~mod~(A_{d^2+1-\rho_d}).\]
\end{coro}
\begin{coro}[Corollary \ref{star}]For $d>0$ and $\chi$ coprime to $d$, $M(d,\chi)$ is stably rational.
\end{coro}

This is our strategy: choose $\chi<0$, then every 1-dimensional sheaf $F$ with Euler characteristic $\chi$ first Chern class $dH$ can be written into the following exact sequence.
\begin{equation}\label{too}0\ra \mo_{\p^2}(-3)\ra \widetilde{I}\ra F\ra0.\end{equation}
If $\widetilde{I}$ is torsion free, then $\widetilde{I}\cong I_{\frac{d(d-3)}2-\chi}(d-3)$ with $I_{\frac{d(d-3)}2-\chi}$ an ideal sheaf of colength $\frac{d(d-3)}2-\chi$, then we get an element in $Hilb^{[\frac{d(d-3)}2-\chi]}(\p^2)$.  However, if $Supp(F)$ is not integral, $\widetilde{I}$ can contain torsion.  Also on the other hand, $F$ in (\ref{too}) with $\widetilde{I}$ torsion free is not necessarily (semi)stable.  Hence we need to estimate codimension of some subschemes (or substacks) in both $M(d,\chi)$ ($\mm(d,\chi)$) and $Hilb^{[\frac{d(d-3)}2-\chi]}(\p^2)$ ($\mh^{\frac{d(d-3)}2-\chi}$).  

The structure of the paper is as follows.  In Section 2, we define some stacks and do the codimension estimate for some relatively easier cases, such as the substack parametrizing sheaves with reducible supports.  Section 3 is the most difficult and complicated part of the paper, where we study the sheaves with support $nC$ for some integral curve $C$ and estimate the codimension of the substack parametrizing those sheaves.  In Section 4, we prove Theorem \ref{main} and some corollaries.  In the end, there is the appendix where we give a whole proof of an important theorem (Theorem \ref{tt}) in Section 3.

\emph{Notations.}  (1) Usually we have $d$ and $\chi$ as integers.  For a sheaf $F$, we denote by $c_1(F)$ the first Chern class of $F$.  $d(F)$ is defined to be the number such that $c_1(F)=d(F)H$, and finally we denote by $\chi(F)$ the Euler characteristic of $F$.

(2) Let $C$ be a curve on a surface $X$.  Let $F$ be a sheaf over $X$.  Then $F(\pm C):=F\otimes\mo_X(\pm C)$.  If moreover $X=\p^2$, $F(n):=F\otimes\mo_{\p^2}(n)$ for any $n\in\mathbb{Z}$.

(3) For two sheaves $F_1,~F_2$ over $X$, $\chi(F_2,F_1):=\sum_{i} (-1)^i dim~\text{Ext}^i(F_2,F_1).$

\textbf{Acknowledgements.}  I was supported by NSFC grant 11301292.  I thank Yi Hu for some helpful discussions.  I also thank Shenghao Sun for the help on stack theory.

%%%%%%%%%%%%%%%%%%%%%%%%%%%%%%%%%%%%%%%%%%%%%%%%%%%%%%%%%%%%%      
\section{Some stacks and codimension estimate.}%%%%%%        Section 2        %%%%%%%%%%%
We are always on $\p^2$ except otherwise stated.  Let $H$ be the hyperplane class on $\p^2$. 
\begin{defn}\label{ff}Given three integers $d>0$, $\chi$ and $a$, 
let $\mm_{\bullet}^a(d,\chi)$ be the (Artin) stack parametrizing sheaves $F$ on $\p^2$ with rank 0, $c_1(F)=dH$, $\chi(F)=\chi$ and satisfying either of the following two conditions.
 
($C_1$) $\forall F'\subset F$, $\chi(F')\leq a$;

($C_2$) $F$ is semistable.
\end{defn}
\begin{defn}Let $\mm(d,\chi)$ be the substack of $\mm_{\bullet}^a(d,\chi)$ parametrizing stable sheaves in $\mm_{\bullet}^a(d,\chi)$.
\end{defn}
\begin{rem}\label{con}(1) In Definition \ref{ff}, if $a\geq\chi>0$, ($C_2$) implies ($C_1$).  But we put ($C_1$) and ($C_2$) together for larger generality.

(2) $\mm(d,\chi)$ has a (coarse) moduli space $M(d,\chi)$.  $M(d,\chi)$ is a fine moduli space if $d$ and $\chi$ are coprime.  We know that $M(d,\chi)$ is irreducible of dimension $d^2+1$ (e.g. see Remark 4.2.10 in \cite{yuan}), hence $\mm(d,\chi)$ is of dimension $d^2$.  
\end{rem}

It is easy to see the boundedness of $\mm^a_{\bullet}(d,\chi)$.  Let $\ts^a(d,\chi):=\mm_{\bullet}^a(d,\chi)-\mm(d,\chi)$.

\begin{prop}\label{dlt}$\ts^a(d,\chi)$ is of codimension $\geq d-1$ in $\mm_{\bullet}^a(d,\chi)$.
\end{prop}
\begin{proof}We prove the lemma by induction on $d$.  If $d=1$, then $\ts^a(d,\chi)=\emptyset$ and there is nothing to prove.

Let $d\geq 1$.  Let $F\in\ts^a(d,\chi)$, then $F$ is strictly semistable or unstable.  Hence we can have the following sequence
\begin{equation}0\ra F_1\ra F\ra F_2\ra0,
\end{equation}
with $F_i\in\mm_{\bullet}^{a_i}(d_i,\chi_i)$ for $i=1,2$, \large{$\frac{\chi_2}{d_2}\leq\frac{\chi_1}{d_1}\leq\frac{a}{d_1}$}, \normalsize and $\text{Ext}^2(F_2,F_1)=0$.  Hence there are finitely many possible choices for $((d_1,\chi_1),(d_2,\chi_2))$,  and we can also find upper bounds for $a_i$ (e.g. $a_1\leq a$ and $a_2\leq (d-1)a$). 

Recall that $\chi(F_2,F_1):=\sum_{i}(-1)^idim~\text{Ext}^i(F_2,F_1)$.
The stack $\mext^1(F_2,F_1)$ has dimension $\leq \chi(F_2,F_1)$, because $\textbf{1}+\text{Hom}(F_2,F_1)$ is contained in the automorphism groups of all elements in $\text{Ext}^1(F_2,F_1)$ as in the following diagram.
\begin{equation}\label{exi}\xymatrix@C=2.5cm{0\ar[r] &F_1\ar[r]
\ar[d]_{Id} 
&F\ar[r]\ar[d]_{\cong}^{\varphi\in\textbf{1}+\text{Hom}(F_2,F_1)} &F_2\ar[r] \ar[d]^{Id}&0 \\
0\ar[r] &F_1\ar[r] &F\ar[r] &F_2\ar[r] &0.}\end{equation}

Hence $dim~\mext^1(F_2,F_1)\leq dim~\text{Ext}^1(F_2,F_1)-dim~\text{Hom}(F_2,F_1)=\chi(F_2,F_1)$ by $\text{Ext}^2(F_2,F_1)=0$. 

By induction assumption we have $dim~\mm_{\bullet}^{a_i}(d_i,\chi_i)=d_i^2$, hence we have $dim~\ts^a(d,\chi)\leq \displaystyle{\max_{d_1+d_2=d}}\{d_1^2+d_2^2+d_1d_2\}=d^2-(d-1)$.  Hence the lemma.\end{proof}

\begin{rem}We only define $\ts^a(d,\chi)$ set-theoretically, but it is enough when talking about codimension.
\end{rem}
\begin{coro}The dimension of $\mm_{\bullet}^a(d,\chi)$ is $d^2$ for all $a$.
\end{coro}

Hence we know that for different $a$,  $\mm_{\bullet}^a(d,\chi)$ are birational and isomorphic in codimension $d-2$.  From now on, usually we won't specify the difference between the numbers $a$ in $\mm_{\bullet}^a(d,\chi)$ but only keep in mind there might exist a difference of dimension at most $d^2-d+1$. 

\begin{defn}For two integers $k>0$ and $i$, we define $\mm_{k,i}^a(d,\chi)$ to be the (locally closed) substack of $\mm_{\bullet}^a(d,\chi)$ parametrizing sheaves $F\in\mm_{\bullet}^a(d,\chi)$ with $h^1(F(i)):=dim~H^1(F(i))=k$ and $h^1(F(n))=0,\forall n>i.$ 
\end{defn}
\begin{rem}\label{lfr}According to Lemma 2.2 in \cite{yth}, for every sheaf $F$ pure of dimension 1 on $\p^2$, there is a direct sum of line bundle $E_F$ uniquely determined by $F$, such that we have the following exact sequence.
\begin{equation}\label{lr}0\ra E_F(-1)\ra E_F\ra F\ra0.
\end{equation}
Moreover $c_1(F)=rk(E_F)H$, $c_1(E_F)=(\chi(F)-rk(E_F))H$, with $rk(E_F)$ the rank of $E_F$. 

Let $E_F=\displaystyle{\oplus_{s=0}^m}\mo_{\p^2}(\alpha_s)^{\oplus \beta_s}$ with $\beta_s>0$ and $\alpha_m>\cdots\alpha_{1}>\alpha_0$.  Then one can easily observe $F\in \mm_{k,i}^a(d,\chi)$ for some $a$ $\Leftrightarrow \alpha_0=-i-2, \beta_0=k$.  

We say that $E_F$ is \textbf{connected} or $F$ is connected if $\alpha_{s+1}=\alpha_s+1,\forall~ 0\leq s\leq m-1.$  If $E_F$ is not connected, for instance $\alpha_{s_0}>\alpha_{s_0-1}+1$, then $F$ contains a subsheaf $F'$ such that $E_{F'}=\displaystyle{\oplus_{s=s_0}^m}\mo_{\p^2}(\alpha_s)^{\oplus \beta_s}$ and $E_{F/F'}=\displaystyle{\oplus_{s=0}^{s_0-1}}\mo_{\p^2}(\alpha_s)^{\oplus \beta_s}$.  
\end{rem}
It is easy to see the following proposition.
\begin{prop}\label{bound}For fixed $(d,\chi,a)$, $\mm_{k,i}^a(d,\chi)$ is empty except for finitely many pairs $(k,i)$.

\end{prop}

\begin{defn}Let $\mn^a(d,\chi)$ be the substack of $\mm_{\bullet}^a(d,\chi)$ parametrizing sheaves in $\mm_{\bullet}^a(d,\chi)$ with integral supports.  Let $\mn^a_{k,i}(d,\chi)=\mn^a(d,\chi)\cap \mm_{k,i}^a(d,\chi).$
\end{defn} 

\begin{rem} (1) It is obvious that $\mn^a(d,\chi)$ ($\mn^a_{k,i}(d,\chi)$ resp.) does not depend on $a$ and hence we write $\mn(d,\chi)$ ($\mn_{k,i}(d,\chi)$ resp.) for short.  Also we see that $\mn(d,\chi)\subset\mm(d,\chi)$. 

(2) Let $N(d,\chi)$ be the image of $\mn(d,\chi)$ in the (coarse) moduli space $M(d,\chi)$.  Since $N(d,\chi)$ contains a family of Jacobians over all smooth curves of degree $d$, we see that $dim~N(d,\chi)\geq d^2+1$.  Hence $dim~\mn(d,\chi)=d^2.$
\end{rem}

\begin{defn}(1) For two integers $l>0$ and $j$, we define $\mw_{l,j}^a(d,\chi)$ to be the (locally closed) substack of $\mm_{\bullet}^a(d,\chi)$ parametrizing sheaves $F\in\mm_{\bullet}^a(d,\chi)$ with  $h^0(F(j)):=dim~H^0(F(j))=l$ and $h^0(F(n))=0,\forall n<j$. 

(2) Let $\mv(d,-\chi)$ be the substack of $\mm_{\bullet}^a(d,\chi)$ parametrizing sheaves in $\mm_{\bullet}^a(d,\chi)$ with integral supports.  Let $\mv_{l,j}(d,\chi)=\mv(d,\chi)\cap \mw_{l,j}^a(d,\chi).$
\end{defn} 
\begin{rem}\label{duke}By sending each sheaf $F$ to its dual $\me xt^1(F,\mo_{\p^2}(-3))$, we get an isomorphism $\mm_{k,i}^a(d,\chi)\xrightarrow{\cong}\mw^{-\chi+a}_{k,-i}(d,-\chi)$, which identifies $\mn_{k,i}(d,\chi)$ with $\mv_{k,-i}(d,-\chi)$.
\end{rem}

\begin{prop}\label{dnki}For $\chi+id\geq0$, $dim~\mn_{k,i}(d,\chi)\leq d^2-(\chi+id)-k$.
\end{prop}
\begin{proof}Denote by $Hilb^{[n]}(\p^2)$ the Hilbert scheme of $n$-points on $\p^2$.  We view $Hilb^{[n]}(\p^2)$ as the moduli scheme of ideal sheaves with colength $n$ and every element in $Hilb^{[n]}(\p^2)$ has automorphism group $\mathbb{C}^{*}$.  Let $\mh^n$ be the stack associated to $Hilb^{[n]}(\p^2)/\mathbb{C}^{*}$.  Then $dim~\mh^n=2n-1$. 

Let $F\in\mn_{k,i}(d,\chi)$, then $H^1(F(i))\neq 0$ and hence we have a non split exact sequence
\begin{equation}\label{hib}0\ra \mo_{\p^2}(-3)\ra I_F(d-3)\ra F(i)\ra 0.
\end{equation} 
Since $Supp(F)$ is integral and (\ref{hib}) does not split, $I_F\in Hilb^{[\tilde{d_i}]}(\p^2)$ with $\tilde{d_i}:=\frac{d(d-3)}2-(id+\chi)$.  

On the other hand, let $I_{\tilde{d_i}}$ be an ideal sheaf of colength $\tilde{d_i}$, let $h\in\text{Hom}(\mo_{\p^2}(-3),I_{\tilde{d_i}}(d-3))$ with $h\neq0$, then $h$ has to be injective.  Let $F_h$ be the cokernel. 
\begin{equation}\label{bih}0\ra\mo_{\p^2}(-3)\xrightarrow{h} I_{\tilde{d_i}}(d-3)\ra F_h\ra0.
\end{equation}
%$F_h$ has to be pure of dimension 1, because $\text{Ext}^1(T,\mo_{\p^2}(-3))=0$ for any 0-dimensional sheaf $T$.  

Denote by $\mh^{\tilde{d_i}}_{(3+i)d+\chi+1}$ the (locally closed) substack of $\mh^{\tilde{d_i}}$ parametrizing ideal sheaves $I_{\tilde{d_i}}$ such that $dim~H^0(I_{\tilde{d_i}}(d))=(3+i)d+\chi+1$.  By (\ref{hib}), $I_F\in\mh^{\tilde{d_i}}_{(3+i)d+\chi+1}$ if $F\in\mn_{k,i}(d,\chi)$.

Let $\mext^1(\mn_{k,i},\mo_{\p^2}(-3))^{*}$ be the stack over $\mn_{k,i}(d,\chi)$ parametrizing non-spliting extensions in $\text{Ext}^1(F(i),\mo_{\p^2}(-3))$ with $F\in\mn_{k,i}(d,\chi)$.  Then 
\[dim~\mext^1(\mn_{k,i},\mo_{\p^2}(-3))^{*}=k+dim~\mn_{k,i}(d,\chi)\]

Let $\mhom(\mo_{\p^2}(-3),\mh^{\tilde{d_i}}_{(3+i)d+\chi+1})^{*}$ be the stack over $\mh^{\tilde{d_i}}_{(3+i)d+\chi+1}$ parametrizing non zero map in $\text{Hom}(\mo_{\p^2}(-3),I_{\tilde{d_i}}(d-3))$ with $I_{\tilde{d_i}}\in \mh^{\tilde{d_i}}_{(3+i)d+\chi+1}$. Then
\[\begin{array}{r}dim~\mhom(\mo_{\p^2}(-3),\mh^{\tilde{d_i}}_{(3+i)d+\chi+1})^{*}=(3+i)d+\chi+1+dim~\mh^{\tilde{d_i}}_{(3+i)d+\chi+1}\\
\leq 2\tilde{d_i}+\chi+(3+i)d=d^2-(\chi+id).\end{array}\]

We then have an injection by (\ref{hib})
\[\mext^1(\mn_{k,i},\mo_{\p^2}(-3))^{*}\hookrightarrow \mhom(\mo_{\p^2}(-3),\mh^{\tilde{d_i}}_{(3+i)d+\chi+1})^{*}.\]
Hence 
\[dim~\mext^1(\mn_{k,i},\mo_{\p^2}(-3))^{*}\leq dim~\mhom(\mo_{\p^2}(-3),\mh^{\tilde{d_i}}_{(3+i)d+\chi+1})^{*},\]
which implies
\[dim~\mn_{k,i}(d,\chi)\leq d^2-(\chi+id)-k.\]
The proposition is proved.
\end{proof}
\begin{rem}\label{duck}By Proposition \ref{dnki} and Remark \ref{duke}, we know that \[dim~\mv_{l,j}(d,\chi)\leq d^2+(\chi+jd)-l,~for~\chi+jd<0.\]
\end{rem}

Let $\ls$ be the linear system of $\mo_{\p^2}(d)$.  Then we have a morphism $\pi:\mm^a_{\bullet}(d,\chi)\ra\ls$ sending every sheaf to its support.  Denote $\ls_o$ the open subscheme of $\ls$ parametrizing all integral curves, $\ls_{r}$ the locally closed subscheme parametrizing sheaves with reducible supports, and finally $\ls_n$ the closed subscheme parametrizing sheaves with irreducible and non-reduced supports, i.e. of form $\frac{d}kC$ for some integral curve $C\in|kH|$.   We have that $\pi^{-1}(\ls_o)=\mn(d,\chi)$ and $\ls=\ls_o\cup\ls_r\cup\ls_n$.  

We want to estimate the codimension of the subset $\mc^a(d,\chi):=\mm_{\bullet}^a(d,\chi)-\mn(d,\chi)$.  Let $\mc^a_r(d,\chi):=\pi^{-1}(\ls_r)$ and $\mc^a_n(d,\chi):=\pi^{-1}(\ls_n)$.   
\begin{lemma}\label{codr}$\mc^a_r(d,\chi)$ is of codimension $\geq d-1$.
\end{lemma}
\begin{proof}We can use the same strategy as in Proposition \ref{dlt}.  Hence it is enough to show that every sheaf $F\in\mc^a_r(d,\chi)$ can be written as an extension of $F_2\in\mm^{a_2}_{\bullet}(d_2,\chi_2)$ by $F_1\in\mm^{a_1}_{\bullet}(d_1,\chi_1)$ with $\text{Ext}^2(F_2,F_1)=0$, and moreover there are finitely many possible choices of $((d_1,\chi_1),(d_2,\chi_2))$ and we can find upper bounds for $a_i$.

Let $C$ be the support of $F\in\mc_r^a(d,\chi)$.  $C$ is reducible, so we can write $C=C_1\cup C_2$ such that $C_1\cap C_2$ is of 0-dimension.  Let $d_i$ be the degree of $C_i$.  Then we have two exact sequences.
\begin{equation}\label{sfot}0\ra\mo_{C_1}(-d_2)\ra\mo_C\ra\mo_{C_2}\ra0;
\end{equation}
\begin{equation}\label{sfto}0\ra\mo_{C_2}(-d_1)\ra\mo_C\ra\mo_{C_1}\ra0.
\end{equation}
Tensor (\ref{sfot}) and (\ref{sfto}) by $F$ and we get
\begin{equation}\label{tsfot}Tor^1(F,\mo_{C_2})\xrightarrow{\jmath_1} F(-d_2)|_{C_1}\xrightarrow{\imath_1} F \ra F|_{C_2}\ra0;
\end{equation}
\begin{equation}\label{tsfto}Tor^1(F,\mo_{C_1})\xrightarrow{\jmath_2} F(-d_1)|_{C_2}\xrightarrow{\imath_2} F\ra F|_{C_1}\ra0.
\end{equation} 
Let $F_i^{tf}$ be the quotient sheaf of $F|_{C_i}$ module its maximal 0-dimensional subsheaf.  Then the image of $\imath_1$ is $F^{tf}_1(-d_2)$, because the image of $\jmath_1$ is supported at $C_1\cap C_2$ and hence a 0-dimensional subsheaf in $F(-d_2)|_{C_1}$ and $F$ is pure.  The same holds for $\imath_2$.  Hence we have
\begin{equation}\label{psfot}0\ra F_1^{tf}(-d_2)\ra F \xrightarrow{p_2} F|_{C_2}\ra0;
\end{equation}
\begin{equation}\label{psfto}0\ra F_2^{tf}(-d_1) \ra F\ra F|_{C_1}\ra0.
\end{equation} 
Compose map $p_2$ with the surjection $F|_{C_2}\ra F_2^{tf}$, we get a sequence as follows.
\begin{equation}\label{fsfot}0\ra F_1\ra F \ra F_2^{tf}\ra0;
\end{equation}
where $F_1$ is the extension of the maximal 0-dimensional subsheaf of $F|_{C_2}$ by $F_1^{tf}(-d_2)$.  Hence $a\geq\chi(F_1)\geq\chi(F_1^{tf}(-d_2))=\chi(F_1^{tf})-d_2d_1$.   The same holds for $F^{tf}_2$ and hence we have $\chi(F_2^{tf})\leq a+d_1d_2\leq a+d^2$.  Moreover for every subsheaf $G\subset F_2^{tf}$, by (\ref{psfto}) $G(-d_1)$ is a subsheaf of $F$, hence $\chi(G(-d_1))=\chi(G)-d(G)d_1\leq a$, and hence $\chi(G)\leq a+d^2$. 

Now (\ref{fsfot}) gives us the extension we need: $F_1\in\mm^a_{\bullet}(d_1,\chi_1)$, $F_2^{tf}\in\mm^{a+d^2}_{\bullet}(d_2,\chi_2)$;  and since $C_1\cap C_2$ is of 0-dimensional and both $F_1$ and $F_2^{tf}$ are pure of dimensional 1, $\text{Hom}(F_1(3),F_2^{tf})=0$ and hence $\text{Ext}^2(F_2^{tf},F_1)=0$.  For fixed $(d,\chi,a)$, there are finitely many possible choices of $((d_1,\chi_1),(d_2,\chi_2))$ because $\chi-a-d^2\leq\chi_1\leq a$.  Hence the lemma.
\end{proof}

The codimension of $\mc^a_n(d,\chi)$ is more complicated to estimate and the result is not so neat as $\mc^a_r(d,\chi)$.  We do it Section 3.

%%%%%%%%%%%%%%%%%%%%%%%%%%%%%%%%%%%%%%%%%%%%%%%%%%%%%%%%%%%%%%
\section{Sheaves with non-reduced supports.}  %%%%%%%%%%%%%         Section 3        %%%%%%%%%%%%%%%%%%%
Sheaves in $\mc^a_n(d,\chi)$ have their supports the form $\frac{d}{k}C$ with $C$ an integral curve with degree $k$.  Let $\mc_k\subset\mc^a_n(d,\chi)$ be the substack parametrizing sheaves with supports $\frac{d}kC$ for $C\in|kH|^o$.  Hence $\mc^a_n(d,\chi)$ is a disjoint union of $\mc_k$ with $k|d$.

%%%%%%%%%%%%%%%%%         Subsection 3.1      %%%%%%%%%%%%%%%%%%
\begin{flushleft}{\textbf{$\dagger$ \large{$\mc_k$ for $k=1,2$.}}}\end{flushleft} 

\begin{prop}\label{ckot}For $k=1,2$, $\mc_k$ is of codimension $\geq d-1$.
\end{prop}
\begin{proof}We use the same strategy again as in Lemma \ref{codr} and Proposition \ref{dlt}, and the proposition follows immediately from the following lemma.
\end{proof}
\begin{lemma}\label{po}Let $F$ be a pure sheaf with support $r C$ on any surface $X$, such that $C\cong\p^1$.  Let $\xi=C.C$ be the self intersection number of $C$.  Assume moreover $\xi\geq0$.  Then $F$ admits a filtration
\[0=F_0\subsetneq F_1\subsetneq\cdots\subsetneq F_r=F,\]
such that $F_{i}/F_{i-1}\cong\mo_{\p^1}(s_i)$ and $s_i-s_{i+1}\geq -\xi$.  Moreover we can ask such filtration also to satisfy that 
\[\forall 0<i\leq r, \text{Hom}(F_i(C),F/F_i)=0.\]   
\end{lemma}
\begin{proof}Since $C\cong\p^1$, every pure sheaf on $C$ is locally free and splits into the direct sum of line bundles.  Now take an exact sequence on $X$
\[0\ra\mo_{C}(s_1)\ra E\ra\mo_C(s_2)\ra0.\] 
We claim that if $s_1<s_2-\xi$, then $E$ is a locally free sheaf of rank 2 on $C$ and hence $E$ splits into direct sum of two line bundles.   

Denote by $\text{Ext}^1_{C}(\mo_C(s_2),\mo_C(s_1))$ the group of extensions of $\mo_C(s_2)$ by $\mo_C(s_1)$ as sheaves of $\mo_C$-modules.  Each sheaf in $\text{Ext}^1_{C}(\mo_C(s_2),\mo_C(s_1))$ is a rank 2 bundle on $C$.  Notice that $\text{Ext}^1_{C}(\mo_C(s_2),\mo_C(s_1))$ is a linear subspace inside $\text{Ext}^1(\mo_C(s_2),\mo_C(s_1))$, since every non-split extension in $\text{Ext}^1_{C}(\mo_C(s_2),\mo_C(s_1))$ is a non-split extension in $\text{Ext}^1(\mo_C(s_2),\mo_C(s_1))$.  So to prove the claim, we only need to show the following statement.
\begin{equation}\label{exd}dim~\text{Ext}^1_{C}(\mo_C(s_2),\mo_C(s_1))=dim~\text{Ext}^1(\mo_C(s_2),\mo_C(s_1)),\forall s_1<s_2-\xi.\end{equation}
The LHS is easy to compute and we get LHS$=dim~H^1(\mo_{\p^1}(s_1-s_2))=s_2-s_1-1$.  Since $\xi\geq0$, $s_1\leq s_2-1$ and hence $s_2-s_1-1$ is a non-negative number.  

$\chi(\mo_C(s_2),\mo_C(s_1))=-C.C=-\xi$ by Hirzebruch-Riemman-Roch on $X$.  

$\text{Hom}(\mo_C(s_2),\mo_C(s_1))=0$ since $s_1\leq s_2-1$.  $dim~\text{Ext}^2(\mo_C(s_2),\mo_C(s_1))=$ $dim~\text{Hom}(\mo_C(s_1),\mo_C(s_2+K_X.C))$ by Serre duality, with $K_X$ the canonical line bundle on $X$.  The canonical line bundle on $C$ is given by $K_X\otimes\mo_X(C)|_C$ and isomorphic to $\mo_{\p^1}(-2)$, hence $K_X.C+C.C=-2$ and hence $K_X.C=-2-\xi.$  Therefore, $dim~\text{Hom}(\mo_C(s_1),\mo_C(s_2+K_X.C))=s_2-s_1-\xi-1\geq 0.$  Finally we have $dim~\text{Ext}^1(\mo_C(s_2),\mo_C(s_1))=s_2-s_1-1$.  Hence (\ref{exd}) holds.

Now we construct a filtration as follows.  We choose $F_1\cong\mo_C(s_1)$ to be the subsheaf supported on $C$ with rank 1 and the maximal degree, i.e. $\forall F'_1\subset F, F'_1\cong\mo_C(s'_1)$, then we have $s'_1\leq s_1$.  Apply induction assumption to $F/F_1$ and we then get a filtration.  It is easy to check that this filtration satisfies the property in the lemma.  Hence we proved the lemma.  
\end{proof}
\begin{rem}(1) Proposition 3.4 in \cite{moz} is a special case for Lemma \ref{po} with $\xi=0$.

(2) For sheaves $F_1$ and $F_2$ supported at an integral curve $C$, $\text{Ext}_C^i(F_1,F_2)$ is in general not a subspace of $\text{Ext}^i(F_1,F_2)$ for $i\geq 2$, i.e. the map $\text{Ext}_C^i(F_1,F_2)\ra\text{Ext}^i(F_1,F_2)$ might not be injective.
\end{rem}

%%%%%%%%%%%%           subsection 3.2 \mc_k in general                  %%%%%%%%%%%%%%%%%%%%%%%%
\begin{flushleft}{\textbf{$\dagger$\large{ $\mc_k$ in general.}}}\end{flushleft} 

\begin{prop}\label{gck}Let $F\in\mc_k$ and let $C$ be the reduced curve in $Supp(F)$, then there is a filtration of $F$ 
\[0=F_0\subsetneq F_1\subsetneq\cdots\subsetneq F_l=F,\]
such that $Q_i:=F_{i}/F_{i-1}$ are torsion-free sheaves on $C$ with rank $r_i$.  $\sum r_i=\frac{d}k$, and moreover there are injections $f^i_F:Q_{i}(-C)\hookrightarrow Q_{i-1}$ induced by $F$ for all $2\leq i\leq l$.
\end{prop}
\begin{proof}Let $\delta_C$ be the function defining the curve $C$.   Since $C$ is integral, $\delta_C$ is irreducible.  For a sheaf $F\in\mc_k$ with reduced support $C$,  $\exists~ l\in\mathbb{Z}_{>0}$ such that $\delta_C^l\cdot F=0$ and $\delta_C^{l-1}\cdot F\neq0.$  Take $F_1$ to be the subsheaf of all the annihilators of $\delta_C$, i.e. $F_1(U):=\{e\in F(U)|\delta_C\cdot e=0\},\forall~ U$ open.  $F_1$ is a pure 1-dimensional sheaf of $\mo_C$-module and hence it is a torsion free sheaf on $C$.  $F/F_1$ is pure of dimension 1, because $F_1$ is the maximal subsheaf of $F$ supported on $C$.
Apply the induction assumption to $F/F_1$,  and we get a filtration $0=F_0\subsetneq F_1\subsetneq\cdots\subsetneq F_l=F$ with $Q_i:=F_{i}/F_{i-1}$ torsion-free on $C$.

We want to show there are injective maps $f_F^i:Q_i(-C)\hookrightarrow Q_{i-1}$.  By induction, it is enough to construct the map $f_F^2:Q_2(-C)\hookrightarrow Q_1$.  We have the following exact sequence.
\begin{equation}\label{ftwo}0\ra Q_1\ra F_2\ra Q_2\ra 0.
\end{equation}
By the definition of $Q_1=F_1$ and $F_2$, we know that $\delta_C\cdot F_2\neq 0$ and $\delta^2_C\cdot F_2=0$.  Hence multiplying $\delta_C$ gives a non-zero map $m_C:F_2(-C)\ra F_2$ with the kernel $Q_1(-C)$ and the image contained in $Q_1$.  Hence $m_C$ induced a injective map $f_F^2:Q_2(-C)\hookrightarrow Q_1$.   Hence the proposition.
\end{proof}

Propositon \ref{gck} implies that we have a morphism from $\mc_k$ to some Flag scheme by sending $F$ to $(Q_l\subset Q_{l-1}(C)\subset\cdots\subset Q_1((l-1)C))$.  But still it is difficult to compute its dimension in general.
\begin{rem}The filtration constructed in the proof of Proposition \ref{gck} is unique.   Hence we stratify $\mc_k$ by the ranks $r_i$ of the factors $Q_i$ as follows.
 \begin{equation}\label{stmc}\mc_k=\displaystyle{\coprod_{\begin{array}{c}r_1\geq \cdots\geq r_l>0,\\ \sum r_i=\frac{d}{k}.\end{array}}} \mc_k^{r_1,\cdots,r_l}.\end{equation}
\end{rem}

\begin{lemma}\label{rko}$\mc_k^{1,1,\cdots,1}$ is of codimension $\geq d-1$ in $\mm^a_{\bullet}(d,\chi)$.
\end{lemma}
\begin{proof}In this case we have $l\cdot k=d$ and $l\geq2$.  It is easy to check for given $(d,\chi,a)$ there are finitely many possible choices for $(d(Q_i),\chi(Q_i))$, where $Q_i$ are the factors in the filtration in Proposition \ref{gck}.  Actually we have $d(Q_i)=1$, $\chi(Q_i)\geq\chi(Q_{i+1})-\frac{d^2}{l^2}$, $\displaystyle{\sum_{i=1}^s}\chi(Q_i)\leq a$ for all $s<l$ and finally $\displaystyle{\sum_{i=t}^l}\chi(Q_i)\geq \chi-a$ for all $t>1$.  By the finiteness of $\{(d(Q_i),\chi(Q_i))\}$, we can estimate the dimension of $\mc^{1,\cdots,1}_k$ for some fixed $(d(Q_i)=1,\chi(Q_i))$. 

With no loss of generality, we assume $k\geq3$.  We first prove the lemma for $l=2$.  Let $F\in\mc_{\frac d2}^{1,1}$.  Then $F$ can be fit in the following sequence.
\begin{equation}\label{let}0\ra Q_1\ra F\ra Q_2\ra 0.
\end{equation}
Let $C$ be the reduced support of $F$.  By Proposition \ref{gck} we have $Q_i$ are torsion free of rank 1 on $C$ and there is an injection $f:Q_2(-C)\hookrightarrow Q_1$.  The parametrizing space of rank 1 torsion free sheaves on $C$ is its compactified Jacobian and well-known to be integral with dimension the arithmetic genus $g_C$ of $C$ (see \cite{aik}).  If there is a number $N$ satisfying that $dim~\text{Ext}^2(Q_2,Q_1)\leq N$ for all $Q_i$ in (\ref{let}) with $F\in \mc^{1,1}_{\frac d2}$, then using analogous argument to Proposition \ref{dlt} we can easily deduce the following estimate.
\begin{equation}\label{eft}dim~\mc_{\frac d2}^{1,1}\leq dim~|\frac{d}2H|+g_C+g_C-\chi(Q_2,Q_1)+N-1.
\end{equation}
$g_C=\frac{(\frac d2-1)(\frac d2-2)}2,$ and $\chi(Q_2,Q_1)=-C.C=-\frac{d^2}4$ by Hirzebruch-Riemann-Roch. 

Now we need to find a suitable $N$ to bound the dimension of $\text{Ext}^2(Q_2,Q_1)$.  We find a upper bound of $dim~\text{Hom}(Q_1(3),Q_2)$.  Since there is an injection from $Q_2(-C)$ to $Q_1$ with cokernel 0-dimensional, $\text{Hom}(Q_1(3),Q_2)$ is a subspace of $\text{Hom}(Q_2(3-C),Q_2)$.  Since $C$ is Gorenstein with dualizing sheaf $\omega_C$ and $\mo_C(-3+C)\cong\omega_C$, we have
%In order to compute the dimension of $\text{Hom}(Q_2(3-\frac d2),Q_2)$, we can assume $\chi(Q_2)=-1$, then $\text{Hom}(Q_2,\omega_C)\neq 0$ with $\omega_C$ the dualizing sheaf on $C$.  Since both $Q_2$  and $\omega_C$ are torsion free of rank 1, any nonzero map $\tau: Q_2\ra \omega_C$ is injective.  Hence $\text{Hom}(Q_2(3-\frac d2),Q_2)$ is a subspace of $\text{Hom}(Q_2(3-\frac d2),\omega_C)=H^1(Q_2(3-\frac d2))^{\vee}$.  One computes that $deg(Q_2(3-\frac d2))=g_C-2-C.C+\frac{3d}2<0$ for $d=2\cdot deg(C)\geq 6$, hence we have the following estimate.
\begin{eqnarray}&dim&\text{Ext}^2(Q_2,Q_1)= dim~\text{Hom}(Q_1(3),Q_2)\nonumber\\
&\leq &dim~\text{Hom}(Q_2(3-C),Q_2)\nonumber\\&=& dim~\text{Hom}(Q_2,Q_2\otimes\omega_C) \nonumber\\
&\leq& deg(\omega_C)+1=\frac {d^2}4-\frac 32d+1.
\end{eqnarray}  
Let $N=\frac {d^2}4-\frac 32d+1$ and (\ref{eft}) gives the following equation.
\begin{equation}\label{efft}dim~\mc_{\frac d2}^{1,1}\leq d^2-(d-1)+(-\frac{d^2}8-\frac {11d}4+1)\leq d^2-(d-1).
\end{equation}
Hence we proved the lemma for $l=2$.

Let $l\geq 3$.  Let $F\in\mc^{1,\cdots,1}_{\frac dl}$ and take the filtration of $F$ as given in Proposition \ref{gck}.  Then we have the following sequence.
\begin{equation}\label{esg}0\ra F_1\ra F\ra F/F_1\ra 0.
\end{equation}
If $\exists~ N$ such that $dim~\text{Hom}(F_1(3),F/F_1)\leq N$ for all $F_1$ in (\ref{esg}) with $F\in \mc^{1,\cdots,1}_{\frac dl}$, then by induction assumption we have the following estimate.
\begin{eqnarray}\label{estg}&dim&\mc_{\frac dl}^{1,\cdots,1(l)}\leq dim~\mc^{1,\cdots,1(l-1)}_{\frac{d}l}+g_C-\chi(F/F_1,F_1)+N\nonumber\\
&\leq & (\frac{l-1}l)^2\cdot d^2-(\frac{l-1}l\cdot d-1)+g_C-\chi(F/F_1,F_1)+N
\end{eqnarray}
The number $l$ in $\mc_{\frac dl}^{1,\cdots,1(l)}$ stands for the number of $1$ in the superscript.  $\chi(F/F_1,F_1)=\frac{(l-1)d}{l}\cdot\frac dl$  by Hirzebruch-Riemann-Roch. 

Notice that any nonzero map $F_1(3)\ra F/F_1$ has its image annihilated by $\delta_C$ and hence contained in $Q_2=F_2/F_1$.  Thus $\text{Hom}(F_1(3),F/F_1)=\text{Hom}(F_1(3),Q_2)$ and then by the same argument as we did for $l=2$, we can let $N$ in (\ref{estg}) to be $\frac{d^2}{l^2}-\frac{3d}l+1$.   Therefore
\begin{eqnarray}\label{esff}&&dim~\mc_{\frac dl}^{1,\cdots,1(l)}\nonumber\\
&\leq& (\frac{l-1}l)^2\cdot d^2-(\frac{l-1}l\cdot d-1)+g_C+\frac{(l-1)d^2}{l^2}+\frac{d^2}{l^2}-\frac{3d}l+1\nonumber\\
&=&d^2-(d-1)+(\frac{3-2l}{2l^2}d^2-\frac{7}{2l}d+2)\leq d^2-(d-1).
\end{eqnarray}
The last inequality is because $l\geq 3$ and $d\geq 3l.$  Hence the lemma.

\end{proof}
\begin{prop}\label{coha}$\mc_{\frac d2}$ is of codimension $\geq d-1$ in $\mm_{\bullet}^a(d,\chi)$.
\end{prop}
\begin{proof}According to the stratification (\ref{stmc}), $\mc_{\frac d2}$ only has two strata: $\mc_{\frac d2}^{1,1}$ and $\mc_{\frac d2}^{2}$.  The former is of codimension $\geq d-1$ by Lemma \ref{rko}.  Hence we only need to estimate $dim~\mc_{\frac d2}^{2}$.  Sheaves in $\mc_{\frac d2}^{2}$ are rank 2 torsion free sheaves on some integral curve $C$ of degree $\frac d2$.  With no loss of generality, we assume $0<\chi\leq d$.  Hence for every sheaf $F$ in $\mc_{\frac d2}^2$ with support $C$, there is a nonzero global section which has to be a injection since both $\mo_C$ and $F$ are torsion free and $C$ is integral.  Hence we have the following sequence.
\begin{equation}\label{had}0\ra\mo_C\ra F\ra \widehat{I}\ra 0.
\end{equation}
The quotient $\widehat{I}$ may not be torsion free.  Take $I_2$ to be the quotient of $\widehat{I}$ module its torsion.  Then we have another exact sequence as follows.
\begin{equation}\label{half}0\ra I_1\ra F\ra I_2\ra0,
\end{equation}
where $I_1$ is a torsion free rank 1 sheaf with non-negative degree.  Let $\chi_i=\chi(I_i)$.  Then we have $1-\frac{(\frac d2-1)(\frac d2-2)}2=\chi(\mo_C)\leq\chi_1\leq a$, hence there are finitely many possible choices for $(\chi_1,\chi_2)$.  Notice that $(\ref{half})$ gives an element in $\text{Ext}^1_C(I_2,I_1)$ which is a linear subspace inside $\text{Ext}^1(I_2,I_1)$.

If there is a number $N$ satisfying that $dim~\text{Ext}^2(I_2,I_1)\leq N$ for all $I_i$ in (\ref{half}) with $F\in \mc^{2}_{\frac d2}$, then using analogous argument we can easily deduce the following estimate.
\begin{equation}\label{dha}dim~\mc_{\frac d2}^{2}\leq dim~|\frac{d}2H|+g_C+g_C-\chi(I_2,I_1)+N-1,
\end{equation}
We can find a suitable $N$ to bound $dim~\text{Ext}^2(I_2,I_1)$ as follows.
\begin{eqnarray}&dim&\text{Ext}^2(I_2,I_1)= dim~\text{Hom}(I_1(3),I_2)\nonumber\\
&\leq &dim~\text{Hom}(\mo_C(3),I_2)= H^0(I_2(-3))\leq deg(I_2(-3))+1\nonumber\\
&\leq& deg(\widehat{I}(-3))=-\frac {3d}2+\chi+2(g_C-1)+1.
\end{eqnarray}  
Let $N=-\frac {3d}2+\chi+2g_C-1$ and (\ref{dha}) gives the following equation.
\begin{equation}\label{dhal}dim~\mc_{\frac d2}^{2}\leq d^2-(d-1)+(-\frac{d^2}8-\frac {11d}4+1+\chi),
\end{equation}
where $-\frac{d^2}8-\frac {11d}4+1+\chi\leq 0$ for $\chi\leq d$ and $d\geq 2$.  Hence the proposition.
\end{proof}

Lemma \ref{codr}, Proposition \ref{ckot} and  Proposition \ref{coha} together give the following proposition.
\begin{prop}\label{prime}For $d=p$ or $2p$ with $p$ a prime number,  the complement of $\mn(d,\chi)$ inside $\mm_{\bullet}^a(d,\chi)$ is of codimension $\geq d-1$. 
\end{prop}

We expect Proposition \ref{prime} holds for all $d$, but still at the moment we only have a much weaker result for other $d$ coming later.

Look back to the filtration in Proposition \ref{gck}.  The maps $f_F^i:Q_i(-C)\hookrightarrow Q_{i-1}$ are all injective but not surjective in general.  Let $\Sigma_i:=F/F_{i-1}$, then $Q_{i}$ is a subsheaf of $\Sigma_i$ and $\delta_C^{l-i+1}\cdot \Sigma_i=0$.  Let $\Pi_i$ be the image of $f^{i+1}_F$ inside $Q_i$.  By the definition of $f_F^i$, one can easily see that $\delta_C^{l-i}\cdot(\Sigma_i/\Pi_i)=0$.  Hence $\Sigma_i/\Pi_i$ is actually supported at $(l-i)C$ and it is just $F\otimes\mo_{(l-i)C}$. 

\begin{prop}\label{bcd}Let $F\in\mc_k$ and let $C$ be the reduced curve in $Supp(F)$, then there is a filtration of $F$ 
\[0=F^0\subsetneq F^1\subsetneq\cdots\subsetneq F^m=F,\]
such that $R_i:=F^{i}/F^{i-1}$ are sheaves on $C$ with rank $t_i$.  $\sum t_i=\frac{d}k$, and moreover there are surjections $g^i_F:R_{i}(-C)\twoheadrightarrow R_{i-1}$ induced by $F$ for all $2\leq i\leq m$.  $R_i$ are not necessarily torsion free.
\end{prop}
\begin{proof}We choose $F^{m-1}$ to be the kernel of the map $F\twoheadrightarrow F\otimes\mo_C$, and hence $R_m\cong F\otimes\mo_C$.  $F^{m-1}$ is the quotient of $F\otimes\mo_{(m-1)C}(-C)$ module the image of $Tor^1(F,\mo_C)$,  hence we have a surjective map $g^m_F:R_m(-C)\twoheadrightarrow R_{m-1}:=F^{m-1}\otimes \mo_C$. We then get the proposition by induction.
\end{proof}

Compare the two filtrations given in Proposition \ref{gck} and Proposition \ref{bcd} and we have the following lemma.
\begin{lemma}\label{cop}Let $(l,r_i)$ and $(m,t_i)$ be as in Proposition \ref{gck} and Proposition \ref{bcd} respectively.  Then we have 

(1) $l=m$;

(2) $r_i=t_{m-i+1}$.
\end{lemma}
\begin{proof}Statement (1) is trivial, since both $m$ and $l$ are the minimal power of $\delta_C$ to annihilate $F$.

We first prove Statement (2) for $l=2$.  Recall that we denote by $\Pi_1$ the image of $f^2_F$ inside $F_1$, and $F/\Pi_1\cong F\otimes \mo_{(l-1)C}$.  Hence for $l=2$ $F/\Pi_1\cong F\otimes\mo_C\cong R_2$.  Hence $t_2=r_2+r_1-r_2=r_1$ and $t_1=r_2$.

Let $l\geq 3$.  Take the torsion free quotient $\widetilde{F}$ of $F/\Pi_1$ and we have $\widetilde{r}_1=r_2+r_1-r_2=r_1$, $\widetilde{r}_i=r_{i+1}$ for $i>1$, and $\widetilde{t}_{m-i}=t_{m-i+1}$ for $i\geq 1$.  Hence by induction assumption, we have $r_1=t_m$, $r_{i+1}=\widetilde{r}_{i}=\widetilde{t}_{m-1-i+1}=t_{m-i+1}$ for $i\geq 2$.  We then have $r_2=t_{m-1}$ because $\sum r_i=\sum t_i$.  Hence the lemma.  
\end{proof}
\begin{defn}\label{uplow}We call the filtration given in Proposition \ref{gck} \textbf{the lower filtration of $F$} while the one given in Proposition \ref{bcd} \textbf{the upper filtration of $F$}.
\end{defn}
\begin{rem}We did not use the assumption that the surface is $\p^2$ in Proposition \ref{gck}, Proposition \ref{bcd} and Lemma \ref{cop}.  Hence they apply to any surface.
\end{rem}
%Recall that $\mm(d,\chi)$ is the moduli space of stable sheaves with parameter $(d,\chi)$.  
Define $\mm(d,\chi)\supset \mt_n:=\{F~|~\exists ~x\in \p^2, s.t.~dim_{k(x)}(F\otimes k(x))\geq n\},$ where $k(x)$ is the residue field of $x$.  In other words, $\mt_n$ is the substack parametrizing sheaves with fiber dimension $\geq n$ at some points. 
\begin{rem}\label{fib}For a sheaf $F$ with filtration in Proposition \ref{gck} or Proposition \ref{bcd}, let $n_0=r_1=t_m$, then we have $F\in\mt_{n_0}$.
\end{rem}
\begin{prop}\label{ttc}For $n\geq 2$, $\mt_n$ is of codimension $\geq n^2-2$ in $\mm(d,\chi)$.
\end{prop}
\begin{proof}Recall that we have a coarse moduli space $M(d,\chi)$ as a scheme.  We denote $T_n$ the image of $\mt_n$ in $M(d,\chi)$.  This proposition is equivalent to say that $T_n$ is of codimension $\geq n^2-2$ in $M(d,\chi)$, which in fact follows straightforward after Le Potier's argument in proving Lemma 3.2 in \cite{lee}.   

We know that there is a Qout-scheme $\Omega(d,\chi)$ such that $\sigma:\Omega(d,\chi)\ra M(d,\chi)$ is a $PGL(V)$-bundle.  By Le Potier's result in \cite{lee}, the preimage $\sigma^{-1}(T_n)$ of $T_n$ is a closed subscheme of codimension $\geq n^2-2$ in $\Omega(d,\chi)$.  It is easy to see that $\sigma^{-1}(T_n)$ is invariant under the $PGL(V)$-action, hence the proposition.    
\end{proof}

By Proposition \ref{ttc} we know that $\mt_3$ is of codimension $\geq 7$.  

Let $\mt_n^o=\mt_n-\mt_{n+1}$.  
\begin{thm}\label{tt}$\mt_2^o\cap \mc_k$ is of codimension $\geq d-1$ in $\mm(d,\chi)$.
\end{thm}
\begin{proof}The proof is too long and moved to Appendix A.
\end{proof}
Finally we get an estimate of dimension for other $d$ as follows.
\begin{prop}\label{gdw}For $d\neq p,2p$ with $p$ a prime number,  the complement of $\mn(d,\chi)$ inside $\mm_{\bullet}^a(d,\chi)$ is of codimension $\geq 7$. 
\end{prop}
\begin{rem}A priori in Proposition \ref{gdw} the lower bound of the codimension should be $\min\{d-1,7\}$.  However when $d-1<7$, $d=p$ or $2p$ for some $p$ prime.  Hence $d-1\geq7$ for all the cases which Proposition \ref{gdw} applies to.
\end{rem}

%%%%%%%%%%%%%%%%%%%%%%%%%%%%%%%%%%%%%%%%%%%%%%%%%%%%%%%%%%%%%
\section{The main theorem.}  %%%%%%%%%%             Section 4       %%%%%%%%%%%%%%
We prove the main theorem in this section.  Recall that we denote by $\mh^n$ the stack associated to the Hilbert schemes $Hilb^{[n]}(\p^2)$ parametrizing ideal sheaves of colength $n$ on $\p^2$.  The strategy is to relate the moduli stack $\mm(d,\chi)$ with $\mh^n$ for some $n$.  First we have two lemmas as follows.
\begin{lemma}\label{quit}Let $J$ be any torsion free rank 1 sheaf on $\p^2$ such that $H^0(J)\neq 0$.  Then any nonzero element  $h_J\in H^0(J)$ gives a sequence
\[0\ra\mo_{\p^2}\xrightarrow{h_J} J\ra F_{h_J}\ra 0,\]
with $F_{h_J}$ pure of dimensional one.
\end{lemma}
\begin{proof}The injectivity of $h_J$ is obvious.  Let $T\subset F_{h_J}$ be 0-dimensional.  Since $\text{Ext}^1(T,\mo_{\p^2})^{\vee}\cong \text{Ext}^1(\mo_{\p^2},T)=0$, $T$ must also be contained in $J$.  Then $T=0$ by the torsion freeness of $J$.  Hence the lemma.
\end{proof}
\begin{lemma}\label{inch}Let $n=\frac{d(d-3)}2+\Delta$ for some $\Delta>0$.  Let $\mh^{n,l}~(0\leq l\leq \frac{d(d-3)}2+1)$ be the substack of $\mh^n$ parametrizing ideal sheaves $I_{n}$ of colength $n$ satisfying that $dim~H^0(I_n(d-3))=l$.  Then for $l>0$, $dim~\mh^{n,l}\leq 2n-1-\Delta$.
\end{lemma}
\begin{proof}For an ideal sheaf $I_n\in\mh^{n,l}$ with $l>0$, we can fit it into the following sequence.
\[0\ra\mo_{\p^2}\ra I_n(d-3)\ra F\ra 0.\]
By Lemma \ref{quit}, $F\in\mm_{\bullet}^a(d-3,-\Delta)$ (with $a=l$ for instance).  Moreover $dim~H^0(F(-3))\leq dim~H^0(F)=l-1.$  Hence $dim~H^1(F(-3))\leq l-1+\Delta+3(d-3).$  Then by analogous argument to the proof of Proposition \ref{dnki}, we have
\[dim~\mh^{n,l}+l\leq dim~\mm_{\bullet}^a(d-3,\Delta)+l-1+3(d-3)+\Delta=2n-1-\Delta+l.\]
Hence the lemma.
\end{proof}

Let $\mu_A(-)$ be the $A$-valued motivic measure (see e.g. Section 1 in \cite{kap}) with $A$ a commutative ring or a field if needed.  Denote by $A_n$ the subgroup (not a subring) generated by the image of $\mu_A(\ms)$ with $dim~\ms\leq n$.  

By Proposition \ref{dlt}, we know that 
\[\mu_A(\mm^a_{\bullet}(d,\chi))~\equiv~\mu_A(\mm(d,\chi))~~~mod~(A_{d^2-d+1}).\]

Define
\begin{equation}\label{rho}\rho_d:=\left\{\begin{array}{l}d-1,~for~d=p~or~2p~with~p~prime.\\ \\ 7,~otherwise.\end{array}\right.\end{equation}

By Proposition \ref{prime} and Proposition \ref{gdw}, we have
\[\mu_A(\mm^a_{\bullet}(d,\chi))~\equiv~\mu_A(\mn(d,\chi))~~~mod~(A_{d^2-\rho_d}).\]

Notice that $\rho_d\leq d-1$.  Let $-2d-1\leq\chi\leq - d+1$, then $-\chi\geq \rho_d$ and $3d+\chi\geq \rho_d$.  For every sheaf $F\in\mm_{\bullet}^a(d,\chi)$, there is a non split sequence
\begin{equation}\label{nsp}0\ra\mo_{\p^2}(-3)\ra \widetilde{I}\ra F\ra0.\end{equation}
$\widetilde{I}$ can have torsion if $F\not\in \mn(d,\chi)$.  If $\widetilde{I}$ is torsion free, then $\widetilde{I}\cong I_{\bar{d}}(d-3)$ for some ideal sheaf $I_{\bar{d}}$ with colength $\bar{d}:=\frac{d(d-3)}2-\chi$.  Let $\um^a(d,\chi)$ be the open substack of $\mm^a_{\bullet}(d,\chi)$ parametrizing sheaves $F$ such that $H^0(F)=0$ and $H^1(F(3))=0$.  Then we have
\[\mm^a_{\bullet}(d,\chi)=\um^a(d,\chi)\cup(\coprod_{j\leq 0} \mw^a_{l,j}(d,\chi)\cup\coprod_{i\geq 3}\mm^a_{k,i}(d,\chi)).\]
Since $\chi+jd\leq\chi\leq-\rho_d<0$ for $j\leq0$ and $\chi+id\geq \rho_d>0$ for $i\geq 3$, by Proposition \ref{duke} and Remark \ref{duck} we have
\[dim~(\coprod_{j\leq 0} \mw^a_{l,j}(d,\chi)\cup\coprod_{i\geq 3}\mm^a_{k,i}(d,\chi))\leq d^2-\min\{\rho_d,-\chi,3d+\chi\}=d^2-\rho_d.\]
Hence
\[\mu_A(\mm^a_{\bullet}(d,\chi))~\equiv~\mu_A(\um^a(d,\chi))~~~mod~(A_{d^2-\rho_d}).\]

Define $\mn_0(d,\chi):=\mn(d,\chi)\cap\um^a(d,\chi)$.  Then 
\begin{equation}\label{mmm}\begin{array}{l}\mu_A(\mm_{\bullet}^a(d,\chi))\equiv\mu_A(\mm(d,\chi))\equiv\mu_A(\um^a(d,\chi))\\ \qquad\qquad\qquad\equiv\mu_A(\mn_0(d,\chi))~mod~(A_{d^2-\rho_d}).\end{array}\end{equation}

On the other hand, by Lemma \ref{inch} we have
\[\mu_A(\mh^{\bar{d},0})~\equiv~\mu_A(\mh^{\bar{d}})~~~mod~(A_{2\bar{d}+\chi-1}).\]

Notice that $\forall I_{\bar{d}}\in\mh^{\bar{d},0}$, $H^0(I_{\bar{d}}(d))\neq0$ since $\chi(I_{\bar{d}}(d))=3d+1+\chi>0.$  Define $\mh^{\bar{d},0,0}$ to be the open substack of $\mh^{\bar{d},0}$ parametrizing ideal sheaves $I_{\bar{d}}\in\mh^{\bar{d},0}$ such that $H^1(I_{\bar{d}}(d))=0$.
\begin{lemma}\label{ddv}$\mh^{\bar{d},0}-\mh^{\bar{d},0,0}$ is of dimension $\leq 2\bar{d}-1-\rho_d$.
\end{lemma}
\begin{proof}$\forall I_{\bar{d}}\in\mh^{\bar{d},0}-\mh^{\bar{d},0,0}$, $H^0(I_{\bar{d}}(d))\neq0$ hence by Lemma \ref{quit} we have the following exact sequence
\[0\ra\mo_{\p^2}(-3)\ra I_{\bar{d}}(d-3)\ra F\ra 0,\]
with $F\in\mm_{\bullet}^{a}(d,\chi)$.  Since $H^0(F)\cong H^0(I_{\bar{d}}(d-3))=0$, $dim~H^1(F)=-\chi$.  Moreover, $H^1(F(3))\cong H^1(I_{\bar{d}}(d))\neq 0$ hence $F\in\coprod_{i\geq3}\mm^a_{k,i}(d,\chi)$.  By Proposition \ref{duke}, $dim~\coprod_{i\geq3}\mm^a_{k,i}(d,\chi)\leq d^2-\min\{(3d+\chi),\rho_d\}=d^2-\rho_d$.  By the analogous argument to the proof of Proposition \ref{dnki} we have
\[dim~(\mh^{\bar{d},0}-\mh^{\bar{d},0,0})+3d+1+\chi\leq d^2-\rho_d-\chi.\]
$2\bar{d}=d(d-3)-2\chi$.  Hence the lemma. 
\end{proof}
Lemma \ref{inch} and Lemma \ref{ddv} together imply that
\begin{equation}\label{mmh}\mu_A(\mh^{\bar{d}})~\equiv~\mu_A(\mh^{\bar{d},0})~\equiv~\mu_A(\mh^{\bar{d},0,0})~~~mod~(A_{2\bar{d}-1-\rho_d}).
\end{equation}

Let stacks $\mext^1(-,\mo_{\p^2}(-3))^{*}$ and $\mhom(\mo_{\p^2}(-3),-)^{*}$ be as defined in the proof of in Proposition \ref{dnki}.  The sequence (\ref{nsp}) induces a birational map
\[\theta:\mext^1(\mm^a_{\bullet}(d,\chi),\mo_{\p^2}(-3))^{*}\dashrightarrow \mhom(\mo_{\p^2}(-3),\mh^{\bar{d}})^{*}.\]
$\theta$ is surjective for $a$ big enough.  

Denote by $\mathbb{U}^a(d,\chi)$ the preimage of $\mhom(\mo_{\p^2}(-3),\mh^{\bar{d},0,0})^{*}$ via $\theta$.  Then we have
\begin{equation}\label{bird}\mu_A(\mathbb{U}^a(d,\chi))=(\bl^{3d+1+\chi}-1)\cdot\mu_A(\mh^{\bar{d},0,0}),\end{equation}
where $\bl:=\mu_A(\mathbb{A})$ with $\mathbb{A}$ the affine line.  Then by (\ref{mmh}) we have
\begin{equation}\label{due}\mu_A(\mathbb{U}^a(d,\chi))~\equiv~(\bl^{3d+1+\chi}-1)\cdot\mh^{\bar{d}}~\equiv~\bl^{3d+1+\chi}\cdot\mh^{\bar{d}}~~mod~(A_{d^2-\chi-\rho_d})
\end{equation}

On the other hand, we have \[\mext^1(\mn_0(d,\chi),\mo_{\p^2}(-3))^{*}\subset \mathbb{U}^a(d,\chi)\subset \mext^1(\um^a(d,\chi),\mo_{\p^2}(-3))^{*}.\]
Hence by (\ref{mmm}),
\begin{equation}\label{une}\begin{array}{l}\mu_A(\mathbb{U}^a(d,\chi))~\equiv~(\bl^{-\chi}-1)\cdot\mu_A(\mn_0(d,\chi))~\\ \qquad\qquad\qquad\equiv~(\bl^{-\chi}-1)\cdot\mu_A(\mm(d,\chi))\\
\qquad\qquad\qquad\equiv~\bl^{-\chi}\cdot\mu_{A}(\mm(d,\chi))~~~mod~(A_{d^2-\chi-\rho_d}).\end{array}
\end{equation}
Combine (\ref{due}) and (\ref{une}), we have our main theorem as follows.
\begin{thm}\label{main}For and $d>0$ and $\chi$, let $\chi_0\equiv \pm\chi~mod~(d)$ and $-\frac {3d}2\leq\chi_0\leq-d$ (such $\chi_0$ is unique).  Then  we have
\[\mu_A(\mm(d,\chi))~\equiv~\bl^{3d+1+2\chi_0}\cdot \mu_A(\mh^{\bar{d}}),~~~mod~(A_{d^2-\rho_d}),\]
with $\bar{d}=\frac{d(d-3)}2-\chi_0$ and $\rho_d$ defined in (\ref{rho}).  

On the scheme level we have
\[\mu_A(M(d,\chi))~\equiv~\bl^{3d+1+2\chi_0}\cdot \mu_A(Hilb^{\bar{d}}(\p^2)),~~~mod~(A_{d^2+1-\rho_d}).\]
\end{thm}
\begin{rem}\label{zero}We choose $-\frac{3d}2\leq\chi_0\leq -d$ in Theorem \ref{main} because we want $\chi_0$ to be uniquely determined by $\chi$.  But it is easy to see Theorem \ref{main} holds for $\chi_0\equiv \pm\chi~mod~(d)$ and $-2d-1\leq\chi_0\leq-d+1$.
\end{rem}
\begin{coro}\label{bet}Let $b_i(-)$ and $h^{p,q}(-)$ be the $i$-th Betti number and Hodge number with index $(p,q)$ respectively.  Then for any $d>0$ and $\chi$ coprime to $d$, if $i$ and $p+q$ are both no less than $1+2(d^2+1-\rho_d)$, we  then have

(1) $b_{i}(M(d,\chi))=0$ for $i$ odd.

(2) $h^{p,p}(M(d,\chi))=b_{2p}(M(d,\chi))=b_{2p-2(3d+1+2\chi_0)}(Hilb^{\bar{d}}(\p^2))$. 

(3) $h^{p,q}=0$ for $p\neq q$.
\end{coro}
\begin{coro}\label{disk}For any $d>0$ and $\chi_1,\chi_1$, we have
\[\mu_A(\mm(d,\chi_1))~\equiv~ \mu_A(\mm(d,\chi_2)),~~~mod~(A_{d^2-\rho_d}).\]
In particular, if $\chi_i$ are coprime to $d$ for $i=1,2$, then we have
 \[\mu_A(M(d,\chi_1))~\equiv~ \mu_A(M(d,\chi_2)),~~~mod~(A_{d^2+1-\rho_d}).\]
\end{coro}
\begin{proof}By Theorem \ref{main} and Remark \ref{zero}, the corollary is equivalent to say that for any $-2d-1\leq \chi_1,\chi_2\leq -d+1$,
\begin{equation}\label{hi}\bl^{3d+1+2\chi_1}\cdot \mu_A(\mh^{\bar{d}_1})~\equiv~\bl^{3d+1+2\chi_2}\cdot \mu_A(\mh^{\bar{d}_2}),~~~mod~(A_{d^2-\rho_d}),\end{equation}
where $\bar{d}_i=\frac{d(d-3)}2-\chi_i$.

It is enough to show (\ref{hi}) for $\chi_1=-2d-1$ and $\chi_2=-d+1$ which follows from $M(d,-2d-1)\cong M(d,-d+1)$.  Hence the corollary.
\end{proof}
\begin{rem}If $d=p$ or $2p$ with $p$ prime, then the codimension $d-1$ can not be sharpened, i.e. in general
\[\mu_A(\mm(d,\chi))~\not\equiv~\bl^{3d+1+2\chi_0}\cdot \mu_A(\mh^{\bar{d}}),~~~mod~(A_{d^2-d}).\]
We can see this from the examples $d=4$ and $d=5$ computed in \cite{yth}. 
\end{rem}

\begin{coro}\label{star}For $d>0$ and $\chi$ coprime to $d$, $M(d,\chi)$ is stably rational.
\end{coro}
\begin{proof}Let $N_0(d,\chi)$ and $Hilb^{\bar{d},0,0}(\p^2)$ be the scheme associated to $\mn_0(d,\chi)$ and $\mh^{\bar{d},0,0}$ respectively.  We can see that the projective bundle $\mathbb{P}(\mathcal{E}xt_p^1(\mf,\mo(-3)))$ over $N_0(d,\chi)$ is birational to the projective bundle $\mathbb{P}(\mathcal{H}om_p(\mo(-3),\mathcal{I}_{\bar{d}}))$ over $Hilb^{\bar{d},0,0}(\p^2)$ which is rational.  The universal sheaf $\mf$ exists by Theorem 3.19 in \cite{lee} and that is why we need $d,\chi$ coprime.  Hence we proved the corollary.
\end{proof}

\begin{rem}By Proposition 4.5 in \cite{yth}, $M(d,\chi)$ is rational for $\chi\equiv \pm1~mod~(d)$.
\end{rem}
\begin{rem}For $d$ and $\chi$ not coprime, let $M^{ss}(d,\chi)$ be the moduli space of semistable sheaves with parameters $(d,\chi)$, then $M^{ss}(d,\chi)-M(d,\chi)$ is not empty.  But the S-equivalence classes of strictly semistable sheaves form a closed subset of codimension $\geq d-1$ in $M^{ss}(d,\chi)$.  Hence we still have
\[\mu_A(M^{ss}(d,\chi))~\equiv~\bl^{3d+1+2\chi_0}\cdot \mu_A(Hilb^{\bar{d}}(\p^2)),~~~mod~(A_{d^2-\rho_d+1}).\]
However, since $M^{ss}(d,\chi)$ might not be smooth, we don't have similar conclusion to Corollary \ref{bet} on its Betti numbers. 
\end{rem}
\begin{rem}Generalization of Theorem \ref{main} to other rational surfaces is certainly possible and we believe our main strategy  works well to other surfaces.  Only one needs to take some effort to estimate the codimension of the subset containing all those``bad" points in the moduli spaces, which could be very tedious and difficult especially when Proposition \ref{ttc} does not hold.  By the proof of Lemma 4.2.7 in \cite{yuan}, Proposition \ref{ttc} holds also for Hirzebruch surfaces $\mathbb{P}(\mo_{\p^1}\oplus\mo_{\p^1}(-e))$ with $e=0,-1$.  Hence one can expect that the generalization to those two surfaces is tractable.   
\end{rem}

%%%%%%%%%%%%%%%%%%%%%%%%%%%%%%%%%%%%%%%%%%%%%%%%%%%%%%%%%%%%%%%%%%%%%%%%%%%%%%%%%%%%%%%%%%%%%%%%%%%%%%%%%%%%%%%%%%%%%%%%%%%%%%%%
\section*{Appendix}  %%%%%%%%%%%%             Appendix          %%%%%%%%%%%%%%
\appendix
\section{The proof of Theorem \ref{tt}.}
We give a whole proof of Theorem \ref{tt} in this section.  We state the theorem again here.
\begin{thm}[Theorem \ref{tt}]$\mt_2^o\cap \mc_k$ is of codimension $\geq d-1$ in $\mm(d,\chi)$.
\end{thm}
\begin{proof}Let $F\in\mt_2^o\cap\mc_k$ with lower and upper filtrations $\{F_i\}$ and $\{F^i\}$ (see Definition \ref{uplow}) with factors $\{Q_i\}$ and $\{R_i\}$ respectively.  Let $m$ be the length of the two filtrations.  Then $t_m=r_1\leq 2$ by Remark \ref{fib}.  If $r_1=2$, then $R_m\cong F\otimes\mo_C$ has to be locally free of rank 2.  Since $g^m_F:R_m\twoheadrightarrow R_{m-1}$ is surjective, $R_{m-1}$ is either of rank 1 or locally free of rank 2 and if $R_{m-1}$ is locally free of rank 2, then $g^m_F$ is an ismorphism.   

We prove the theorem case by case.  

\emph{Case 1.} $r_1=1$.  Then by Lemma \ref{rko} we are done. 

\emph{Case 2.} $r_i=2$ for all $1\leq i\leq m$.  %%%%%%%%%    Case 2 ,  r_i=2 for all i     %%%%%%%%%%

By induction if $F\in \mc_k^{2,\cdots,2}\cap\mt_2^o$, then $R_i\cong Q_i\cong R_m(-(m-i)C)$ and the two filtrations coincide with all  factors locally free of rank 2.  In this case $k=\frac d{2m}$.  Let $R:=R_m$.  Then $c_1(R)=\frac dm$ and we have 
\begin{equation}\label{chir}\sum_{i=0}^{m-1}\chi(R_m(-iC))=m\cdot\chi(R_m)-\frac{(m-1)m}4\cdot(\frac{d}{m})^2=\chi.\end{equation}
Hence $\chi(R)$ is fixed by $(d,\chi,k)$.  For every subsheaf $I$ of $R$, we have $\chi(R)-\chi(I)\geq \chi-a.$  Let $\mr$ be the parametrizing stack of such $R$.  We first show that 
\begin{eqnarray}\label{dpr}dim~\mr&\leq &\frac{d^2}{m^2}-(\frac dm-1)-(\frac 18(\frac dm)^2+\frac 34\frac dm)\nonumber\\
&= &\frac{d^2}{m^2}-(\frac dm-1)-(\frac 18(\frac {(1\cdot 0+1)d^2}{m^2}+\frac 34\frac dm).\end{eqnarray}

We assume $0<\chi\leq d$, then we have the following exact sequence.
\begin{equation}\label{two}0\ra\mo_C\ra R\ra I_2\ra0.
\end{equation}
We are done then by the same argument as in Lemma \ref{coha} and (\ref{dhal}) implies (\ref{dpr}). 

Now we want to use induction. Let $\mathcal{P}_{F/F_1}$  be the parametrizing stack of $F/F_1=F/R(-(m-1)C)$.  Then by induction assumption we have
\begin{eqnarray}\label{pff} dim~\mathcal{P}_{F/F_1}&\leq&
\frac{d^2(m-1)^2}{m^2}-(\frac{(m-1)d}m-1)\nonumber\\&&-(\frac 18\frac {((m-1)(m-2)+1)d^2}{m^2}+\frac 34\frac {(m-1)d}{m}).\end{eqnarray}  

$dim~\text{Ext}^2(F/F_1,F_1)=dim~\text{Hom}(F_1(3),F_2/F_1)=dim~Hom(R,R(-3+C))$.   We want to find a upper bound $N$ of $dim~Hom(R, R(-3+C)).$  Notice that $dim~Hom(R,R(-3+C))$ won't change if we replace $R$ by $R(n)$ for any number $n$.  

Since $F$ is stable, it is connected by Remark \ref{lfr}.  After replace $F$ by $F(n)$ for some suitable number $n$ we can assume $H^1(F(-1))=0$ and $H^1(F(-2))\neq0$.  Hence we know that the smallest degree of direct summands in $E_F$ must be no bigger than $0$ by Remark \ref{lfr}, and hence $\chi=\chi(F)\leq\frac{(d+1)d}2$ by the connectedness of $F$.  We see that  $H^1(R(-1))=0$ because $R$ is a quotient of $F$ and they are both of 1-dimensional.   Therefore by Mumford-Castelnuovo criterion (see e.g. Lemma 1.7.2 in \cite{hl}) $R$ is globally generated.  Hence $R$ fits in the following sequence.
\begin{equation}\label{gge}0\ra\mo_C\ra R\ra det(R)\ra 0.
\end{equation}
Hence $det(R)$ is also globally generated and $H^1(det(R)(-1))=0$.
Tensor (\ref{gge}) by $\mo_{\p^2}(-3+C)$ and we get
\begin{equation}\label{gee}0\ra \omega_C\ra R(-3+C)\ra det(R)(-3+C)\ra 0.
\end{equation}

Notice that $\mo_C(-3+C)\cong\omega_C$.  The functor $\text{Hom}(R,-)$ sends (\ref{gee}) into the following sequence.
\begin{equation}\label{homr}0\ra \text{Hom}(R,\omega_C)\ra \text{Hom}(R,R(-3+C))\ra \text{Hom}(R,det(R)(-3+C)).
\end{equation}
$\text{Hom}(R,\omega_C)=H^1(R)^{\vee}=0$ and hence
\[dim~ \text{Hom}(R,R(-3+C))\leq  dim~ \text{Hom}(R,det(R)(-3+C)).\]

The functor $\text{Hom}(-,det(R)(-3+C))$ sends (\ref{gge}) into the following sequence.
\begin{equation}\label{hodr}0\ra H^0(\mo_C(-3+C))\ra \text{Hom}(R,det(R)(-3+C))\ra H^0(det(R)(-3+C)).
\end{equation}  

$dim~H^0(\mo_C(-3+C))=g_C$.  Since $C$ is of degree at least 3 and $det(R)$ is globally generated, 
$dim~H^0(det(R)(-3+C))=\chi(det(R))+(-\frac{3d}{2m}+\frac{d^2}{4m^2})=\chi(R)+(g_C-1)+\frac{d^2}{4m^2}-\frac{3d}{2m}.$  
Hence we have 
\begin{eqnarray}\label{gr}dim~\text{Hom}(R,R(-3+C))&\leq& dim~\text{Hom}(R,det(R)(-3+C))\nonumber\\
&\leq& dim~H^0(\omega_C)+dim~H^0(det(R)(-3+C))\nonumber\\
&=&\chi(R)+2g_C-1+\frac{d^2}{4m^2}-\frac{3d}{2m}.
\end{eqnarray}
By (\ref{chir}) and $\chi\leq\frac{d(d-1)}2$, we get $\chi(R)\leq \frac{(d+1)d}{2m}+\frac{m-1}{4}\cdot (\frac d{m})^2$ 

By (\ref{gr}), we have
\begin{eqnarray}dim~\text{Hom}(R, R(-3+C))\leq \chi(R)+2g_C-1+\frac{d^2}{4m^2}-\frac{3d}{2m}\qquad\qquad&&\nonumber\\
\leq \frac{(d+1)d}{2m}+\frac{(m-1)d^2}{4m^2}+(\frac{d}{2m}-1)(\frac d{2m}-2)-1+\frac{d^2}{4m^2}-\frac{3d}{2m} =:N&&
\end{eqnarray}

Then we have 
\begin{eqnarray}\label{fcot}dim~\mc^{2,\cdots,2}_k\cap\mt_2^o&\leq& dim~\mathcal{P}_{F/F_1}+N-\chi(F/F_1,F_1)\nonumber\\
&\leq& \frac{d^2(m-1)^2}{m^2}-(\frac{(m-1)d}m-1)\nonumber\\
&&-(\frac 18\frac {(m^2-3m+3)d^2}{m^2}+\frac 34\frac {(m-1)d}m)+N+\frac{d^2(m-1)}{m^2}\nonumber\\
&\leq&  d^2-(d-1)-(\frac18\frac{(m^2-m+1)d^2)}{m^2}+\frac34d)+(1-\frac34\frac dm)\nonumber\\
&\leq&  d^2-(d-1)-(\frac18\frac{(m^2-m+1)d^2)}{m^2}+\frac34d).
\end{eqnarray}

In particular, the codimension of $\mc^{2,\cdots,2}_k\cap\mt_2^o$ is $\geq d-1$.  

Now we compute the codimension of  $\mc^{2,\cdots,2,1,\cdots,1}_k\cap\mt_2^o$.  We do the induction on the number $\ell(1)$ of $1$ in the superscript of $\mc^{2,\cdots,2,1,\cdots,1}_k$.   

\emph{Case 3.} $\ell(1)=1$.  %%%%%%%%%      Case 3,  l(1)=1       %%%%%%%%%

Let $F\in \mc^{2,\cdots,2,1}_k\cap\mt_2^o$.  Let $C$ be its reduced support with $deg(C)=k=\frac d{2m-1}$ with $m\geq2$.  We take the lower and upper filtrations $\{F_i\}$ and $\{F^i\}$ of $F$ with factors $\{Q_i\}$ and $\{R_i\}$ for $1\leq i\leq m$.  Then $R_m$ is a rank 2 bundle on $C$, $R_{i}\cong R_m((-m+i)C)$ for $2\leq i\leq m$ and $R_1$ is a rank 1 torsion free sheaf on $C$ with surjection $g^2_F:R_2(-C)\twoheadrightarrow R_1$.  Let $K$ be the kernel of $g^2_F$, then $K$ is torsion free of rank 1 and the subsheaf $F_1$ in the lower filtration lies in the following sequence.
\begin{equation}\label{suf}0\ra R_1\ra F_1\ra K(C)\ra 0.
\end{equation} 
For $m\geq 3$, we also have
\begin{equation}\label{sut}0\ra R_1(C)\ra F_2/F_1\ra K(2C)\ra 0.
\end{equation} 
By the stability of $F$, we know that
 \begin{equation}\label{kk}\frac{\chi(F_1)}{2deg(C)}=\frac {\chi(R_1)+\chi(K)+(\frac {d}{2m-1})^2}{\frac {2d}{2m-1}}\leq\frac {\chi}d.\end{equation}
 \begin{equation}\label{kat}\frac{\chi(F_2)}{4deg(C)}=\frac {2\chi(R_1)+2\chi(K)+4(\frac {d}{2m-1})^2}{\frac {4d}{2m-1}}\leq\frac {\chi}d,~for~m\geq3.\end{equation}
(\ref{kk}) and (\ref{kat}) imply that 
\begin{equation}\label{kb}\chi(R_1)+\chi(K)\leq \frac{\chi}{2m-1}-(\frac{d}{2m-1})^2.
\end{equation}
\begin{equation}\label{ka}\chi(R_1)+\chi(K)\leq \frac{2\chi}{2m-1}-2(\frac{d}{2m-1})^2,~for~m\geq3.
\end{equation}
Since $R_1$ is a quotient of $R_2(-C)$, $R_1((m-1)C)$ is a quotient of $R_m$ hence a quotient of $F$.  So 
\begin{equation}\label{cro}\frac{\chi(R_1)+\frac{(m-1)d^2}{(2m-1)^2}}{\frac d{2m-1}}\geq\frac {\chi}d\Leftrightarrow \chi(R_1)\geq \frac{\chi}{2m-1}-\frac{(m-1)d^2}{(2m-1)^2}.\end{equation}
Combine (\ref{kb}), (\ref{ka}) and (\ref{cro}), then we get 
\begin{equation}\label{drb}\chi(K)-\chi(R_1)\leq -\frac{d^2}{(2m-1)^2}+\frac{2(m-1)d^2}{(2m-1)^2},
\end{equation}
\begin{equation}\label{drk}\chi(K)-\chi(R_1)\leq -\frac{2d^2}{(2m-1)^2}+\frac{2(m-1)d^2}{(2m-1)^2}, ~for~m\geq3.
\end{equation}
We need a upper bound for $dim~\text{Ext}^2(F/R_1,R_1)=dim~\text{Hom}(R_1(3),F/R_1)$.  The upper and lower filtrations of $F/R_1$ coincide.  Hence $\text{Hom}(R_1(3),F/R_1)=\text{Hom}(R_1(3),R_2)$.  Then we have 
\begin{eqnarray}\label{rrg}&&dim~\text{Ext}^2(F/R_1,R_1)=dim~\text{Hom}(R_1(3),R_2)\nonumber\\
&\leq& dim~\text{Hom}(R_1,R_1(-3+C))+dim~\text{Hom}(R_1,K(-3+C))\nonumber\\
&\leq& 4g_C-2+\chi(K)-\chi(R_1).
\end{eqnarray} 
By (\ref{drb}) and (\ref{drk}) we have
\begin{equation}\label{lei}dim~\text{Ext}^2(F/R_1,R_1)\leq N:=\left\{\begin{array}{l}-\frac{d^2}{(2m-1)^2}+\frac{2(m-1)d^2}{(2m-1)^2}+4g_C-2,~for~m=2.\\
\\-\frac{2d^2}{(2m-1)^2}+\frac{2(m-1)d^2}{(2m-1)^2}+4g_C-2,
~for~m\geq3.\end{array}\right.
\end{equation}

Let $\mathcal{P}_{F/R_1}$ be the parametrizing stack of $F/R_1$.  We first assume $m\geq3$.  By (\ref{fcot}) and Proposition \ref{dlt}, we know that \begin{equation}\label{dfro}dim~\mathcal{P}_{F/R_1}\leq (\frac {(2m-2)d}{2m-1})^2-(\frac {(2m-2)d}{2m-1}-1).\end{equation}  

Hence by standard argument we have 
\begin{eqnarray}\label{exo}&dim&\mc^{2,\cdots,2,1}_k\cap\mt_2^o\leq dim~\mathcal{P}_{F/R_1}+g_C-1+N-\chi(R_1,F/R_1)\nonumber\\
&\leq&(\frac {(2m-2)d}{2m-1})^2-\frac {(2m-2)d}{2m-1}+g_C+N+\frac{(2m-2)d^2}{(2m-1)^2}\nonumber\\
&=& d^2-(d-1)+(-\frac12 d^2-\frac {7d}{2m-1}+3)\leq d^2-(d-1).
\end{eqnarray}
For $m=2$, by (\ref{dpr}) and (\ref{lei}) we have 
\begin{eqnarray}\label{ext}&dim&\mc^{2,1}_k\cap\mt_2^o\leq dim~\mathcal{P}_{R_2}+g_C-1+N-\chi(R_1,R_2)\nonumber\\
&\leq&(\frac 23d)^2-\frac {2}{3}d-\frac12(\frac d3)^{2}-\frac 12d+g_C+N+\frac{2d^2}{3^2}\nonumber\\
&=& d^2-(d-1)+(-\frac{17d}6+2)\leq d^2-(d-1).
\end{eqnarray}
Notice that one needs to replace $\frac dm$ in (\ref{dpr}) by $\frac23d$ to get the right formula.

We are done for $\ell(1)=1$.

\emph{Case 4: The last case.} $\ell(1)\geq 2$.   %%%%%%%   Case 4  l(1)\geq 2  the last case   %%%%%%%%

Let $F\in\mc_k^{2,\cdots,2,1,\cdots,1}\cap\mt_2^o $ with $\ell(1)\geq 2$.   Let $m_i=\ell(i)$ for $i=1,2$.  Let $C$ be the reduced support of $F$.  Then $deg(C)=\frac{d}{m_1+2m_2}\geq 3$.  By doing the upper filtration, we can write $F$ into the following sequence
\begin{equation}0\ra F'\ra F\ra F''\ra 0,
\end{equation}  
with $F'\in\mc_k^{1,\cdots,1'}$ and $F''\in\mc_k^{2,\cdots,2''}\cap \mt_2^{o''}.$  Those spaces with $'$ and $''$ are analogous spaces to $\mc_k$ and $\mt_2^o$ but with parameters $(d(F'),\chi(F'))$ and $(d(F''),\chi(F''))$ respectively.   

Take the upper and lower filtrations of $F'$ with graded factors $\{R'_i\}$ and $\{Q'_i\}$.  Then both $R'_i$ and $Q'_i$ are of rank 1.  Denote by $R'^{tf}_i$ the torsion free quotient of $R'_i$ module its torsion.  The surjection $g_{F'}^i:R'_i(-C)\ra R'_{i-1}$ identifies $R'^{tf}_i(-C)$ with $R'^{tf}_{i-1}$.  Moreover $Q'_{m_1}=R'^{tf}_{m_1}$.  $Q'_{i-1}$ is an extension of a 0-dimensional sheaf by $Q'_i$ and hence $\chi(Q'_{i})\leq\chi(Q'_{i-1})$. 

We know that the upper and lower filtrations of $F''$ coincide.  Let $R''_i$ be the factors.  Then $\{R''_i,R'_i\}$ is the set of graded factors of the upper filtration for $F$ and hence we have a surjection $g_F^{m_1+1}:R''_1(-C)\twoheadrightarrow R'_{m_1}$.  Hence we have a surjection $p^1_{m_1+1}:R''_1(-C)\twoheadrightarrow Q'_{m_1}$ as $Q'_{m_1}$ is a quotient of $R'_{m_1}$.  Let %$\widehat{K}_{m_1}$ and 
$K_{m_1}$ be the kernel of %$g_{F}^i$ and 
$p^1_{m_1+1}$.
%\begin{equation}\label{kqrh}0\ra \widehat{K}_{m_1}\ra R''_1(-C)\ra R'_{m_1}\ra0.
%\end{equation}
\begin{equation}\label{kqro}0\ra K_{m_1}\ra R''_1(-C)\ra Q'_{m_1}\ra0.
\end{equation}
%Then we have \[0\ra\widehat{K}_{m_1}\ra K_{m_1}\ra T'_{m_1}\ra0.\]
Denote by %$\widehat{P}_{m_1}$ and 
$P_{m_1}$ the subsheaf of $F/F'_{m_1-1}$ given by the following extension.
%\begin{equation}\label{pqrh}0\ra Q'_{m_1}\ra \widehat{P}_{m_1}\ra \widehat{K}_{m_1}(C)\ra 0.\end{equation}
\begin{equation}\label{pqro}0\ra Q'_{m_1}\ra P_{m_1}\ra K_{m_1}(C)\ra 0.
\end{equation}
Then $P_{m_1}$ is a $\mo_C$-module, i.e. it is a rank 2 torsion free sheaf on $C$.  This is because $p^1_{m_1+1}$ is defined by acting $\delta_C$ on $F/F'_{m_1-1}$ and $K_{m_1}$ is the kernel which implies $\delta_C\cdot P_{m_1}=0.$  Moreover,  $P_{m_1}$ is the maximal subsheaf of $F/F'_{m_1-1}$ annihilated by $\delta_C$, since $Q'_{m_1}$ is torsion free of rank 1.   

Again we have a map $p^1_{m_1}:P_{m_1}(-C)\ra Q'_{m_1-1}$ inducing the injection $f_{F'}^{m_1}:Q'_{m_1}(-C)\hookrightarrow Q'_{m_1-1}.$  However, the map $p^1_{m_1}$ is not necessarily surjective and we denote by $S'_{m_1-1}(-C)$ its image in $Q'_{m_1-1}$.  We have $Q'_{m_1}(-C)\subset S'_{m_1-1}(-C)\subset Q'_{m_1-1}$.  

Let $K_{m_1-1}$ be the kernel of $p^1_{m_1}$, then 
\begin{equation}\label{kcc}\chi(K_{m_1})+\chi(Q'_{m_1}(-C))-\chi(Q'_{m_1-1})\leq\chi(K_{m_1-1})\leq \chi(K_{m_1}).
\end{equation}   

Again we have a subsheaf $P_{m_1-1}$ of $F/F'_{m_1-2}$ such that $P_{m_1-1}$ is a rank 2 torsion free sheaf on $C$ lying in the following exact sequence.
\begin{equation}\label{pqrt}0\ra Q'_{m_1-1}\ra P_{m_1-1}\ra K_{m_1-1}(C)\ra0.
\end{equation}
By (\ref{kcc}), we have 
\begin{equation}\label{ppcc}\chi(P_{m_1})-C.C\leq\chi(P_{m_1-1})\leq \chi(P_{m_1})-C.C+\chi(Q'_{m_1-1})-\chi(Q'_{m_1}(-C)).\end{equation} 

We repeat this procedure and finally we get
\begin{equation}\label{pqrl}0\ra Q'_{1}\ra P_{1}\ra K_1(C)\ra0.
\end{equation}
$\chi(P_1)\geq \chi(P_{m_1})-(m_1-1)C.C=\chi(R''_1)-m_1C.C$ by (\ref{ppcc}), (\ref{kqro}), (\ref{pqro}) and induction assumption on $P_{i}$ for $i>1$.  

It is easy to see $P_1=F_1$ with $\{F_i\}$ the lower filtration of $F$.  By the stability of $F$, we have 
\begin{equation}\label{gcro}\chi(Q'_1)\geq \chi(Q'_{m_1})-\frac{(m_1-1)d^2}{(m_1+2m_2)^2}\geq \frac{\chi}{m_1+2m_2}-\frac{(m_1+2m_2-1)d^2}{(m_1+2m_2)^2}.\end{equation}
\begin{equation}\label{gcrt}\chi(Q'_1)+\chi(K_1)\geq \chi(R''_{m_2})-\frac{(m_1+2m_2-1)d^2}{(m_1+2m_2)^2}\geq \frac{2\chi}{m_1+2m_2}-\frac{(m_1+2m_2-1)d^2}{(m_1+2m_2)^2}.\end{equation}

On the other hand, $Q'_1$ is a subsheaf of $F$.  Hence $\chi(Q'_1)\leq \frac{\chi}{(m_1+2m_2)}$, then by (\ref{ppcc}) we have 
\begin{equation}\label{ubk}\chi(K_1)\geq \frac{\chi}{m_1+2m_2}-\frac{(m_1+2m_2-1)d^2}{(m_1+2m_2)^2}.
\end{equation}

If $m_2=1$, then $F/P_1\in\mc_k^{1,\cdots,1}$ and by Lemma \ref{rko} the parametrizing stack $\mathcal{P}_{F/P_1}$ has dimension $\leq (\frac {m_1d}{m_1+2})^2-(\frac{m_1d}{m_1+2}-1)$.  On the other hand, $\text{Hom}(P_1(3),F/P_1)=\text{Hom}(P_1(3),S'_2(-C))\subset\text{Hom}(P_1(3),Q'_2).$
\begin{eqnarray}\label{grrg}&&dim~\text{Ext}^2(F/F_1,F_1)\leq dim~\text{Hom}(F_1(3),Q'_2)\nonumber\\
&\leq& dim~\text{Hom}(Q'_1(3),Q''_2)+dim~\text{Hom}(K_1(C+3),Q'_2)\nonumber\\
&\leq&dim~\text{Hom}(Q'_1(3),Q'_1(C))+dim~\text{Hom}(K_1(C+3),Q'_1(C))\nonumber\\
&\leq& 4g_C-2+\chi(Q'_1)-\chi(K_1(C))\nonumber\\
&\leq& 4g_C-2-\frac{d^2}{(m_1+2m_2)^2}+\frac{(m_1+2m_2-1)d^2}{(m_1+2m_2)}=:N.
\end{eqnarray} 
(\ref{dpr}) gives a upper bound for the dimension of the parametrizing stack of $P_1$.  By using analogous estimate to (\ref{ext}), we proved the case $m_2=2$.

Let $m_2\geq2$.  Then we start the previous procedure again with the surjective map $p^2_{m_1+1}:R''_{2}(-C)\xrightarrow{\cong} R''_{1}\twoheadrightarrow Q'_{m_1}(C)$.  Let $L_{m_1}$ be the kernel of $p^2_{m_1+1}$, then $L_{m_1}\cong K_{m_1}(C)$.  
Define $B_{m_1}$ analogously to $P_{m_1}$ and it lies in the following sequence.
\[0\ra Q'_{m_1}\ra B_{m_1}\ra L_{m_1}(C)\ra0.\]

Then we have a map $p^2_{m_1+1}:B_{m_1}(-C)\ra S'_{m_1-1}(-C)$.  Notice that we have $S'_{m_1-1}(-C)$ instead of $Q'_{m_1-1}$, with $S'_{m_1-1}(-C)$ the image of $p^1_{m_1}$ and $Q'_{m_1}(-C)\subset S'_{m_1-1}(-C)\subset Q'_{m_1-1}$.  Denote by $S'_{i-1}(-C)$ the image of $p^1_{i}$ for $2\leq i\leq m_1$.  We have that $\chi(Q'_{i}(-C))\leq\chi( S'_{i-1}(-C))\leq\chi( Q'_{i-1}).$  

Let $B_{m_1-1}$ be the kernel of $p^2_{m_1}$.  We then get $\{B_i\}$ and $\{L_i\}$ inductively analogous to $\{P_i\}$ and $\{K_i\}$.  We also have    
\begin{equation}\label{tcc}\chi(L_{m_1})+\chi(Q'_{m_1}(-C))-\chi(S'_{m_1-1})\leq\chi(L_{m_1-1})\leq \chi(L_{m_1}).
\end{equation}   
\begin{equation}\label{tpcc}\chi(B_{m_1})-C.C\leq\chi(B_{m_1-1})\leq \chi(B_{m_1})-C.C+\chi(S'_{m_1-1})-\chi(Q'_{m_1}(-C)).\end{equation}

Finally we get the maximal subsheaf $B_1$ of $F/F_1$ annihilated by $\delta_C$.  $B_1$ lies in the following sequence 
\begin{equation}\label{pqrbl}0\ra S'_{1}\ra B_{1}\ra L_1(C)\ra0.
\end{equation}
The injection $f^2_F:B_1(-C)\hookrightarrow P_1$ induces the injection $S'_{1}(-C)\hookrightarrow Q'_1$.  

By induction assumption on $\ell(1)$, we have 
\begin{equation}\label{time}dim~\mathcal{P}_{F/Q'_1}\leq \frac {(m_1+2m_2-1)^2d^2}{(m_1+2m_2)^2}-(\frac{(m_2+2m_2-1)d}{m_1+2m_2}-1).\end{equation}
Let $P'_1$ be the maximal subsheaf of $F/Q'_1$ annihilated by $\delta_C$.  Then we have the following sequence.
\[0\ra K_1(C)\ra P'_1\ra \widetilde{S}'_1\ra 0.\]
$\widetilde{S}'_1(-C)$ is the preimage of $Q'_1$ via the injection $f_F^2$ and hence $S'_1\subset \widetilde{S}'_1\subset Q'_1(C)$.  

By induction, (\ref{kcc}) and (\ref{tcc}) imply the following two equations respectively.
\begin{equation}\label{gkcc}\chi(K_{m_1})+\chi(Q'_{m_1}((1-m_1)C))-\chi(Q'_{1})\leq\chi(K_{1})\leq \chi(K_{m_1}).
\end{equation}  
\begin{equation}\label{gtcc}\chi(L_{m_1})+\chi(Q'_{m_1}((1-m_1)C))-\chi(S'_{1})\leq\chi(L_{1})\leq \chi(L_{m_1}).
\end{equation}

(\ref{ppcc}) and (\ref{tpcc}) imply the following two equations respectively.\small
\begin{equation}\label{gppcc}\chi(P_{m_1})-(m_1-1)C.C\leq\chi(P_1)\leq \chi(P_{m_1})-(m_1-1)C.C+\chi(Q'_1)-\chi(Q'_{m_1}((1-m_1)C)).\end{equation} 
\begin{equation}\label{gtpcc}\chi(B_{m_1})-(m_1-1)C.C\leq\chi(B_1)\leq \chi(B_{m_1})-(m_1-1)C.C+\chi(S'_{1})-\chi(Q'_{m_1}((1-m_1)C)).\end{equation}
\normalsize

Notice that $B_1$ is a subsheaf of $F/P_1$.  Let $\eta:=\chi(Q'_1)-\chi(Q'_{m_1}((1-m_1)C))$.  Recall that $L_{m_1}\cong P_{m_1}(C)$.  Since $m_2\geq 2$, by (\ref{gkcc}), (\ref{gtcc}), (\ref{gppcc}), (\ref{gtpcc}), (\ref{pqrl}), (\ref{pqrbl}) and stability of $F$, we get the following formula analogous to (\ref{ka}).  
\begin{eqnarray}\label{gdrk}&&2\chi(P_1)-2\eta+\frac{2d^2}{m_1+2m_2}\leq\chi(P_1)+\chi(B_1)\leq\frac{4\chi}{m_1+2m_2}%-\frac{2d^2}{m_1+2m_2}
\nonumber\\
&\Rightarrow& \chi(K_1)+\chi(Q'_1)\leq\frac{2\chi}{m_1+2m_2}-\frac{2d^2}{(m_1+2m_2)^2}+\eta.
\end{eqnarray}
By (\ref{gcro}) we know
\[\chi(Q'_1)=\chi(Q'_{m_1}((1-m_1)C))+\eta\geq\frac{\chi}{m_1+2m_2}-\frac{(m_1+2m_2-1)d^2}{(m_1+2m_2)^2}+\eta.\]
We then have 
\begin{eqnarray}\label{kqd}\chi(K_1)-\chi(Q'_1)&\leq& \frac{2(m_1+2m_2-1)d^2}{(m_1+2m_2)^2}-\frac{2d^2}{(m_1+2m_1)^2}-\eta\nonumber\\&\leq& \frac{2(m_1+2m_2-1)d^2}{(m_1+2m_2)^2}-\frac{2d^2}{(m_1+2m_1)^2}.
\end{eqnarray}
On the other hand, we have
\begin{eqnarray}\label{rrgf}&&dim~\text{Ext}^2(F/Q'_1,Q'_1)=dim~\text{Hom}(Q'_1(3),P'_1)\nonumber\\
&\leq& dim~\text{Hom}(Q'_1,K_1(-3+C))+dim~\text{Hom}(Q'_1,Q'_1(-3+C))\nonumber\\
&\leq& 4g_C-2+\chi(K_1)-\chi(Q'_1).
\end{eqnarray} 

Now combine (\ref{time}), (\ref{kqd}) and (\ref{rrgf}), use the same estimate as we used in Case 3 for $m\geq 3$, and at last we get an analogous formula to (\ref{exo}).  The last case is done. 

The theorem is proved.  
\end{proof}

Yao YUAN \\
Yau Mathematical Sciences Center, Tsinghua University, \\
Beijing 100084, China\\
E-mail: yyuan@mail.math.tsinghua.edu.cn.


\begin{thebibliography}{99}
\bibitem{aik}
A.B. Altman, A. Iarrobino, and S. L. Kleiman, \emph{Irreducibility of the compactified Jacobian}, 1-12, Sijthoff and Noordhoff, Alphen aan de Rijn, 1977.
\bibitem{jmdmm}
J-M. Dr\'ezet and M. Maican, \emph{On the geometry of the moduli spaces of semi-stable sheaves supported on plane quartics}, Geom. Dedicata, 152 (2011), 17-49.
\bibitem{hl}
D. Huybrechts,  and M. Lehn.  \emph{The Geometry of Moduli Spaces of Sheaves}.   Friedr.  Vieweg \& Sohn Verlagsgesellschaft mbH,  Braunschweig/Wiesbaden,  1997.
\bibitem{kkv}
S. Katz, A. Klemm and C. Vafa,  \emph{M-Theory, Topological Strings and Spinning Black Holes}, Adv. Theor. Math. Phys. 3 (1999) 1445-1537.

\bibitem{lee}
J. Le Potier,  \emph{Faisceaux Semi-stables de dimension $1$ sur le plan projectif},  Rev. Roumaine Math.  Pures Appl.,  38 (1993),7-8, 635-678.  
\bibitem{kap}
M. Kapranov, \emph{The elliptic curve in the S-duality theory an Eisenstein series for Kac-Moody groups,} 2000, arXiv: math/0001005.
\bibitem{moz}
S. Mozgovoy, \emph{Invariants of Moduli spaces of stable sheaves on ruled surfaces,} arXiv:1302.4134.
\bibitem{pt}
R.Pandharipande and R.P.Thomas, \emph{Curve counting via stable pairs in the derived category}, Invent. Math. 178(2009), 407-447.

\bibitem{toda}
Y. Toda,  \emph{Curve counting theories via stable objects I. DT/PT correspondece}, J. Amer. Math. Soc. 23 (2010), 1119-1157. 
\bibitem{ky}
K. Yoshioka,  \emph{Moduli spaces of stable sheaves on abelian surfaces},  Math. Annalen,  321 (2001), 817-884.

\bibitem{yyf}
Yao Yuan,  \emph{Estimates of sections of determinant line bundles on Moduli spaces of pure sheaves on algebraic surfaces},  Manuscripta Math., Vol. 137, Issue 1 (2012), pp. 57-79.


\bibitem{yuan}
Y.Yuan,  \emph{Determinant line bundles on Moduli spaces of pure sheaves on rational surfaces and Strange Duality},  Asian J. Math. Vol 16, No. 3, pp.451-478, September 2012. 

\bibitem{yth} 
Y. Yuan,  \emph{Moduli spaces of semistable sheaves of dimension 1 on $\mathbb{P}^2$},  Pure Appl. Math. Q., Vol. 10, No 4, 2014, pp 723-766.

%\bibitem{yfour}
%Y. Yuan, \emph{Affine pavings for moduli spaces of pure
%sheaves on $\mathbb{P}^2$ with degree $\leq 5$}, Geometriae
%Dedicata accepted, arXiv: 1312.7117


\end{thebibliography}
\end{document}